\documentclass[12pt,reqno]{amsart}
\usepackage{amsfonts,amssymb}
\usepackage{color}
\usepackage{graphicx}
\usepackage{dsfont}

\usepackage{a4wide}

\makeatletter
\def\eqnarray{\stepcounter{equation}\let\@currentlabel=\theequation
\global\@eqnswtrue
\tabskip\@centering\let\\=\@eqncr
$$\halign to \displaywidth\bgroup\hfil\global\@eqcnt\z@
  $\displaystyle\tabskip\z@{##}$&\global\@eqcnt\@ne
  \hfil$\displaystyle{{}##{}}$\hfil
  &\global\@eqcnt\tw@ $\displaystyle{##}$\hfil
  \tabskip\@centering&\llap{##}\tabskip\z@\cr}

\def\endeqnarray{\@@eqncr\egroup
      \global\advance\c@equation\m@ne$$\global\@ignoretrue}

\def\@yeqncr{\@ifnextchar [{\@xeqncr}{\@xeqncr[5pt]}}
\makeatother

\newtheorem{theorem}{Theorem}[section]
\newtheorem{lemma}[theorem]{Lemma}

\newtheorem{proposition}[theorem]{Proposition}

\theoremstyle{definition}

\numberwithin{equation}{section}

\newtheorem{definition}[theorem]{Definition}
\newtheorem{assu}[theorem]{Assumption}

\newtheorem{rem}[theorem]{Remark}

\def\N{{\mathbb N}}
\def\Z{{\mathbb Z}}

\def\R{{\mathbb R}}

\def\C{{\mathbb C}}

\newcommand{\G}{\text{G}}

\newcommand{\g}{\text{g}}

\renewcommand{\L}{\mathbb{L}}

\newcommand{\tr}{\text{tr}}

\renewcommand{\Re}{\hbox{\rm Re}\,}

\newcommand{\supp}{\text{\rm supp\,}}

\newcommand{\1}{\mathds{1}}
\newcommand{\dv}{\text{div}}

\DeclareMathOperator{\dom}{Dom}    

\newcommand{\dyn}{{\text{d}}}

\newcommand{\Gam}{D}

\newcommand{\dist}{\text{dist}}

\newcommand{\ft}{\mathfrak{t}}

\newcommand{\J}{\mathfrak{J}}

\newcommand{\SSS}{\mathcal{S}}

\def\typeout#1{\message{^^J}\message{#1}\message{^^J}}
\newif\ifSRCOK \SRCOKtrue
\newcount\PAGETOP
\newcount\LASTLINE
\global\PAGETOP=1
\global\LASTLINE=-1
\def\EJECT{\SRC\eject}
\def\WinEdt#1{\typeout{:#1}}
\gdef\MainFile{\jobname.tex}
\gdef\CurrentInput{\MainFile}
\newcount\INPSP
\global\INPSP=0
\def\SRC{\ifSRCOK%
  \ifnum\inputlineno>\LASTLINE%
    \ifnum\LASTLINE<0%
      \global\PAGETOP=\inputlineno%
    \fi%
    \global\LASTLINE=\inputlineno%
    \ifnum\INPSP=0%
      \ifnum\inputlineno>\PAGETOP%
        
      \fi%
    \else%
      
    \fi%
  \fi%
\fi}
\def\PUSH#1{%
\SRC%
\ifnum\INPSP=0 \global\let\INPSTACKA=\CurrentInput \else%
\ifnum\INPSP=1 \global\let\INPSTACKB=\CurrentInput \else%
\ifnum\INPSP=2 \global\let\INPSTACKC=\CurrentInput \else%
\ifnum\INPSP=3 \global\let\INPSTACKD=\CurrentInput \else%
\ifnum\INPSP=4 \global\let\INPSTACKE=\CurrentInput \else%
\ifnum\INPSP=5 \global\let\INPSTACKF=\CurrentInput \else%
               \global\let\INPSTACKX=\CurrentInput \fi\fi\fi\fi\fi\fi%
\gdef\CurrentInput{#1}%
\WinEdt{<+ \CurrentInput}%
\global\LASTLINE=0%
\ifSRCOK\fi%
\global\advance\INPSP by 1}
\def\POP{%
\ifnum\INPSP>0 \global\advance\INPSP by -1  \fi%
\ifnum\INPSP=0 \global\let\CurrentInput=\INPSTACKA \else%
\ifnum\INPSP=1 \global\let\CurrentInput=\INPSTACKB \else%
\ifnum\INPSP=2 \global\let\CurrentInput=\INPSTACKC \else%
\ifnum\INPSP=3 \global\let\CurrentInput=\INPSTACKD \else%
\ifnum\INPSP=4 \global\let\CurrentInput=\INPSTACKE \else%
\ifnum\INPSP=5 \global\let\CurrentInput=\INPSTACKF \else%
               \global\let\CurrentInput=\INPSTACKX \fi\fi\fi\fi\fi\fi%
\WinEdt{<-}%
\global\LASTLINE=\inputlineno%
\global\advance\LASTLINE by -1%
\SRC}
\def\INPUT#1{\relax}

\let\originalxxxeverypar\everypar
\newtoks\everypar
\originalxxxeverypar{\the\everypar\expandafter\SRC}
\everymath\expandafter{\the\everymath\expandafter\SRC}
\output\expandafter{\expandafter\SRCOKfalse\the\output}

\newif\ifSRCOK \SRCOKtrue
\newcount\PAGETOP
\newcount\LASTLINE
\global\PAGETOP=1
\global\LASTLINE=-1
\gdef\MainFile{\jobname.tex}
\gdef\CurrentInput{\MainFile}
\newcount\INPSP
\global\INPSP=0
\def\EJECT{\SRC\eject}
\def\WinEdt#1{\typeout{:#1}}
\def\SRC{\ifSRCOK%
  \ifnum\inputlineno>\LASTLINE%
    \ifnum\LASTLINE<0%
      \global\PAGETOP=\inputlineno%
    \fi%
    \global\LASTLINE=\inputlineno%
    \ifnum\INPSP=0%
      \ifnum\inputlineno>\PAGETOP%
      \fi%
    \else%
    \fi%
  \fi%
\fi}
\def\PUSH#1{%
\SRC%
\ifnum\INPSP=0 \global\let\INPSTACKA=\CurrentInput \else%
\ifnum\INPSP=1 \global\let\INPSTACKB=\CurrentInput \else%
\ifnum\INPSP=2 \global\let\INPSTACKC=\CurrentInput \else%
\ifnum\INPSP=3 \global\let\INPSTACKD=\CurrentInput \else%
\ifnum\INPSP=4 \global\let\INPSTACKE=\CurrentInput \else%
\ifnum\INPSP=5 \global\let\INPSTACKF=\CurrentInput \else%
               \global\let\INPSTACKX=\CurrentInput \fi\fi\fi\fi\fi\fi%
\gdef\CurrentInput{#1}%
\WinEdt{<+ \CurrentInput}%
\global\LASTLINE=0%
\ifSRCOK\fi%
\global\advance\INPSP by 1}
\def\POP{%
\ifnum\INPSP>0 \global\advance\INPSP by -1  \fi%
\ifnum\INPSP=0 \global\let\CurrentInput=\INPSTACKA \else%
\ifnum\INPSP=1 \global\let\CurrentInput=\INPSTACKB \else%
\ifnum\INPSP=2 \global\let\CurrentInput=\INPSTACKC \else%
\ifnum\INPSP=3 \global\let\CurrentInput=\INPSTACKD \else%
\ifnum\INPSP=4 \global\let\CurrentInput=\INPSTACKE \else%
\ifnum\INPSP=5 \global\let\CurrentInput=\INPSTACKF \else%
               \global\let\CurrentInput=\INPSTACKX \fi\fi\fi\fi\fi\fi%
\WinEdt{<-}%
\global\LASTLINE=\inputlineno%
\global\advance\LASTLINE by -1%
\SRC}
\def\INPUT#1{\relax}
\let\OldINCLUDE=\include
\def\include#1{
\EJECT%
\PUSH{#1.tex}%
\OldINCLUDE{#1}%
\POP}

\let\originalxxxeverypar\everypar
\newtoks\everypar
\originalxxxeverypar{\the\everypar\expandafter\SRC}
\everymath\expandafter{\the\everymath\expandafter\SRC}
\let\zzzxxxbibliography=\bibliography
\def\bibliography#1{\PUSH{\jobname.bbl}\zzzxxxbibliography{#1}\POP}
\output\expandafter{\expandafter\SRCOKfalse\the\output}

\begin{document}
\author{Karoline Disser}
\address{Weierstrass Institute, Mohrenstr. 39, 10117 Berlin, Germany}
\email{karoline.disser@wias-berlin.de}
\author{Martin Meyries}
\address{Martin-Luther-Universit\"at Halle-Wittenberg, Institut f\"ur Mathematik, 06099 Halle (Saale), Germany}
\email{martin.meyries@mathematik.uni-halle.de}
\author{Joachim Rehberg}
\address{Weierstrass Institute, Mohrenstr. 39, 10117 Berlin, Germany}
\email{joachim.rehberg@wias-berlin.de}

\title[Mixed boundary conditions and diffusion on Lipschitz interfaces]
{A unified framework for parabolic equations with mixed boundary conditions and \\diffusion on interfaces}



\thanks{K.D. was supported by ERC-2010-AdG no. 267802 ``Analysis of Multiscale 
Systems Driven by Functionals''. M.M. was partially supported by DFG-project ME 3848/1-1 
and would like to thank the Weierstrass Institute for its kind hospitality during a research stay.}

\maketitle

\begin{abstract}
In this paper we consider scalar parabolic equations in a general non-smooth
 setting emphasizing interface conditions and mixed boundary conditions. 
In particular, we allow for dynamics and diffusion on a Lipschitz interface
 and on the boundary, where the diffusion coefficients are only assumed 
to be bounded, measurable and positive semidefinite. In the bulk, we 
additionally take into account diffusion coefficients which may degenerate
towards a Lipschitz surface. For this problem class, we introduce a unified 
functional analytic framework based on sesquilinear forms and 
show maximal $L^p$-regularity and bounded $H^\infty$-calculus for the corresponding operator. 
\end{abstract}

\section{Introduction}
This paper presents a unified framework for a general class of linear 
inhomogeneous mixed initial-boundary value problems of the form
\begin{alignat}{3}
\zeta \partial_t u - \dv(\mu_\Omega \nabla u) & =  f_{\Omega}  & \qquad & \text{in }\,J
\times (\Omega\setminus\Sigma), \label{e-parabol}\\
u & =  0  & &\text {on }\,  J \times \Gamma_D, \label{e-Diri}\\
\nu \cdot \mu_\Omega \nabla u & = 0 && \text{on }\, J \times \Gamma_N, 
\label{e-Neumann}\\ 
\zeta \partial_t u - \dv_{\Gamma_\dyn} (\mu_{\Gamma_\dyn} \nabla_{\Gamma_\dyn} u)
 + \nu \cdot \mu_\Omega \nabla u  & =   f_{\Gamma_{\dyn}} &  & 
\text{on }\, J \times \Gamma_\dyn, \label{e-robin}\\
\zeta \partial_t u - \dv_\Sigma (\mu_\Sigma \nabla_\Sigma u)+ 
 [\nu_\Sigma \cdot \mu_\Omega \nabla u]  & =   f_\Sigma   && \text{on }\,
 J \times \Sigma, \label{e-xi-eq}\\
u(0) & =  u_0  & &\text{in }\, (\Omega\setminus \Sigma) \times \Gamma_{\dyn}
\times \Sigma. \label{e-initial}
\end{alignat}
Here $J=(0,T)$ is a time interval and $\Omega \subset \R^d$ is a bounded 
domain with boundary $\partial \Omega$ and with outer unit 
normal vector field $\nu$.
The boundary is disjointly decomposed into a closed Dirichlet part $\Gamma_D$, 
a Neumann part $\Gamma_N$ 
and a dynamic part $\Gamma_\dyn$, i.e.,  
$$\partial\Omega = \Gamma_D \, \dot{\cup} \, \Gamma_N \, \dot{\cup}\, \Gamma_\dyn.$$ 
Moreover, $\Sigma \subset \Omega$ is a $(d-1)$-dimensional hypersurface with unit 
normal vector field $\nu_\Sigma$, 
on which a further dynamic condition is imposed, and $[\nu_\Sigma \cdot \mu_\Omega \nabla u]$
 denotes the jump of 
$\nu_\Sigma \cdot \mu_\Omega \nabla u$ across $\Sigma$. The surface gradients on 
$\Gamma_\dyn$ and on $\Sigma$ are denoted by
$\nabla_{\Gamma_\dyn}$ and  $\nabla_{\Sigma}$. Accordingly, we write $\dv_{\Gamma_\dyn}$ 
and $\dv_\Sigma$ for the surface divergences,
  such that $\Delta_{\Gamma_\dyn} = \dv_{\Gamma_\dyn} \nabla_{\Gamma_\dyn}$ and 
$\Delta_\Sigma = \dv_{\Sigma}\nabla_\Sigma$ are the 
  Laplace-Beltrami operators. The diffusion coefficients $\mu_\Omega$, 
$\mu_{\Gamma_\dyn}$ and $\mu_\Sigma$ are matrix-valued, 
  and the relaxation coefficent $\zeta$ is positive, bounded, and bounded away from 
zero. The external forces $f_{\Omega}$, $f_{\Gamma_{\dyn}}$
  and $f_\Sigma$ as well as the initial data $u_0$ are assumed to be given.
Initial data have to be prescribed at $\Omega\setminus \Sigma$, $\Gamma_{\dyn}$
 and $\Sigma$ due to the corresponding dynamic equations on these sets.

 \begin{figure}[h]
    \centering
\includegraphics[scale=0.57]{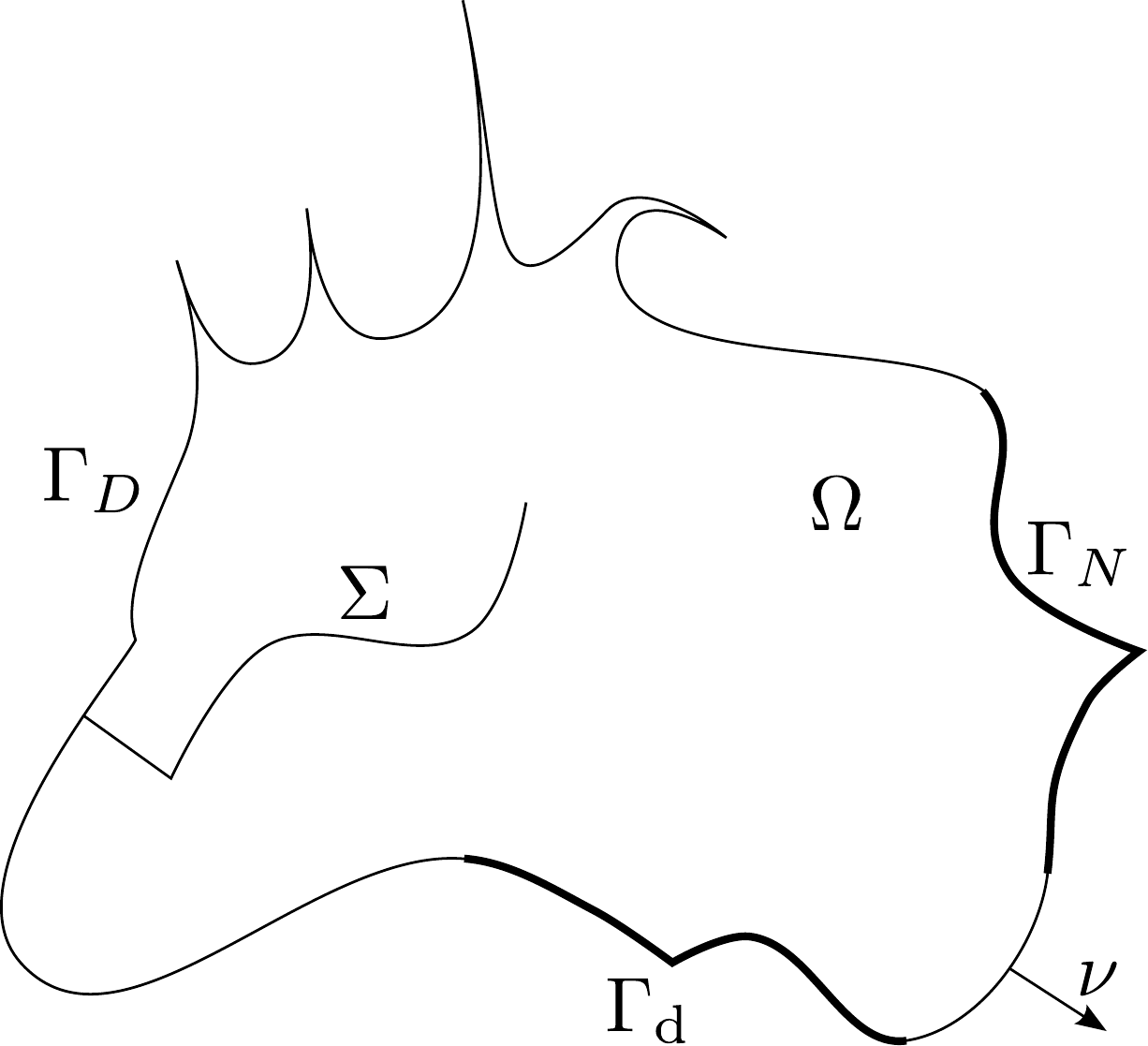}
\caption{Example of a domain $\Omega$ with interface $\Sigma$ and boundary
  $\partial \Omega = \Gamma_D \, \dot{\cup} \, \Gamma_N \, \dot{\cup}\, \Gamma_\dyn$}
    \label{fig:Bild}
\end{figure}

Well-posedness and qualitative properties of parabolic problems with dynamic 
boundary conditions are well-studied, 
  see for example \cite{AIMT, AQRB, BC, EMR, Esc, FGGR, Gal1, Gal2, Hin, HKR, IK, MMS, SW, VaVi, VV}.
  Here, mostly the case of a smooth boundary is considered. Nonlinear degeneracy
 in the diffusion is treated in
 \cite{AIMT, Gal1, IK}. Mixed boundary conditions on non-smooth domains and 
dynamical Robin conditions are also treated in 
 \cite{Nittka, Nittka2}, in a setting which may include inhomogeneities in the
 Neumann or Dirichlet parts.
 Mixed Dirichlet-Wentzell boundary conditions with a smooth Wentzell boundary are
 treated in \cite{VV}.\\
The present paper extends the results of \cite{EMR} in two directions: we consider surface
 diffusion on Lipschitz boundaries and interfaces with diffusion coefficients which
 may degenerate arbitrarily, and further allow the bulk diffusion coefficients to 
degenerate moderately towards another Lipschitz hypersurface. In addition, we still take into account mixed boundary
 conditions nonsmooth diffusion and relaxation coefficients. 
Inhomogeneous Neumann boundary conditions, as well as boundary parts and interfaces evolving in time
are not included in our approach, compare \cite{Nittka2}. We say that the diffusion is "degenerate",
 if the coefficient matrices
 $\mu_\Omega, \mu_{\Gamma_{\dyn}}, \mu_\Sigma$ are not strongly elliptic. In fact, we only require
 $\mu_{\Gamma_{\dyn}}, \mu_\Sigma$ to be non-negative,
 and thus surface diffusion may be absent or degenerate in a very general sense. The 
bulk coefficient matrix $\mu_\Omega$ may also degenerate
 but must still imply bulk regularity of the solution which allows for a trace function
 at $\Gamma_{\dyn}, \Sigma$. Examples of this situation
are given below. 

We present a unified setting based on recent abstract results for sesquilinear forms from \cite{AE}, 
which handles all these nonsmooth scenarios and their combinations at once. 

Let us give more details on the assumptions for the geometry and the coefficients. 
The boundary parts $\Gamma_D$, $\Gamma_N$ and $\Gamma_{\dyn}$ are allowed to meet,
    and also the interface $\Sigma$ may meet any of the boundary parts 
$\Gamma_D, \Gamma_N, \Gamma_\dyn$. 
 Except at points close to the remainder of $\partial \Omega$, no conditions on the Dirichlet part $\Gamma_D$ are imposed. 
    
    The diffusion coefficients $\mu_\Omega$, $\mu_{\Gamma_{\dyn}}$ and $\mu_\Sigma$
 do not have to be symmetric and are 
    assumed to be measurable, bounded and non-negative. To describe their degeneracies in a precise
 way, we assume  pointwise estimates of the form
$$( \mu(x) \xi,\xi) \geq c_1\mu^*(x)|\xi|^2, \quad \xi\in \R^{d},
 \qquad \|\mu(x)\|_{\mathcal L(\R^{d})} \leq c_2 \mu^*(x),
$$
where $\mu$ stands for $\mu_\Omega$, $\mu_{\Gamma_{\dyn}}$ or $\mu_\Sigma$, 
respectively, and $\mu^*$ is in each case a measurable, 
      bounded and nonnegative function. Regarding surface diffusion, we may 
allow for \emph{arbitrary} supports of 
      $\mu_{\Gamma_\dyn}^*$ and $\mu_{\Sigma}^*$. This is to be expected as the well-posedness
 of equations
      \eqref{e-parabol}--\eqref{e-initial} should not depend
      on the presence of surface diffusion. However, it is a considerable part of our
 work to give a suitable definition of surface 
      gradients which captures the exact presence of diffusion on arbitrary subsets
 of the surface, which may still yield regularization 
      where diffusion is present and which still allows us to show maximal regularity
 of the abstract Cauchy problem.

Concerning bulk diffusion, our setting is naturally more restrictive and we only 
consider a class of examples of degenerate diffusion.
For the function $\mu_\Omega^*$, we assume that
\begin{equation}\label{dist}
\mu_\Omega^*(x) = \dist(x,S)^\gamma, \qquad x\in \Omega,
  \end{equation}
where $S\subset \overline{\Omega}$ is an arbitrary $(d-k)$-dimensional Lipschitz 
submanifold of $\R^d$, $1\leq k \leq d$,
      and the exponent is in the range $0<\gamma < k$, which makes $\mu_\Omega^*$ a
 Muckenhoupt weight of class $\mathcal{A}_2$. 
      Of particular interest is the case when $S \cap (\overline{\Gamma_{\dyn} \cup \Sigma})
 \neq \emptyset$,
      i.e., when diffusion degenerates towards $\Gamma_{\dyn}$ or $\Sigma$, but may or may
 not occur along $\Gamma_{\dyn}$ or $\Sigma$. 
     In general, in this case we will have to assume that $\gamma <1$.

We describe the setting in which \eqref{e-parabol}--\eqref{e-initial} is realized. 
The basis of the approach is the sesquilinear form 
\begin{equation} \label{e-form0}
\mathfrak t(u,v) = \int_\Omega (\mu_\Omega \nabla u, \overline{\nabla v})\,dx 
+ \int_{\Gamma_\dyn} \big (\mu_{\Gamma_\dyn}\nabla_{\Gamma_\dyn} u, 
\overline{\nabla_{\Gamma_\dyn} v}\big) \,d\mathcal H_{d-1} 
+ \int_{\Sigma} \big(\mu_\Sigma\nabla_\Sigma u, \overline{\nabla_\Sigma v}\big) 
\,d\mathcal H_{d-1},
\end{equation}
where $\mathcal H_{d-1}$ denotes the $(d-1)$-dimensional Hausdorff measure. 
The surface gradients $\nabla_{\Gamma_\dyn}$ and $\nabla_\Sigma$ on the Lipschitz 
surfaces $\Gamma_\dyn$ and $\Sigma$ are introduced in a simple,
    straightforward way in terms of local coordinates, such that the definitions 
coincide with the corresponding well-known objects in 
    a smooth situation (see Section \ref{sec:SG}). In order to obtain a suitable weak
 formulation of \eqref{e-parabol}--\eqref{e-initial},
    we define the domain of the form $\ft$ as the completion of 
$$
    C_D^\infty(\Omega) :=  \big\{ u|_{\Omega} : u \in C_c^{\infty}({\R}^d), \;
 (\supp u )\cap \Gamma_D
  = \emptyset \big\},
    $$ with respect to 
$$
 \|u\|_{\dom(\ft)}^2 := \|u\|_{W^{1,2}(\Omega,\mu_\Omega^*)}^2 + 
\|\nabla_{\Gamma_{\dyn}} u\|_{L^2(\Gamma_{\dyn}, \mu_{\Gamma_{\dyn}}^*)}^2 +
 \|\nabla_{\Sigma} u\|_{L^2(\Sigma, \mu_{\Sigma}^*)}^2.
$$
Here, $W^{1,2}(\Omega,\mu_\Omega^*)$ is a Sobolev space with weight 
$\mu_\Omega^*$ in the gradient norm, 
     and $L^2(\Gamma_{\dyn}, \mu_{\Gamma_{\dyn}}^*)$ and $L^2(\Sigma, \mu_{\Sigma}^*)$ 
     are Lebesgue spaces equipped with the weights $\mu_{\Gamma_{\dyn}}^*$ and $\mu_\Sigma^*$.

Based on the results of \cite{AE}, to the form $\ft$, we associate an operator
 $A_2$ on the Lebesgue space
\[
\L^2 = L^2\big((\Omega \setminus \Sigma) \cup \Gamma_{\dyn} \cup \Sigma, (dx+d\mathcal H_{d-1})\big)
= L^2(\Omega\setminus \Sigma)\oplus L^2(\Gamma_{\dyn}) \oplus L^2(\Sigma).
\]
In order to realize this setting, one must make sure that for every 
$v \in \dom(\ft)$, there are traces $\tr_\Sigma\, v \in L^2(\Sigma)$ and
$\tr_{\Gamma_{\dyn}} \,v \in L^2(\Gamma_{\dyn})$ such that we obtain a triple $(v,v_\Sigma ,v_{\Gamma_{\dyn}}) \in \L^2$, where
here and in the following, we often use the notation $v_{\mathcal{M}}$ to indicate the restriction or trace of $v$ on a set $\mathcal{M}$ if it is well-defined.  
The constitutive relation for $A_2u$ is then given by
\begin{equation} \label{e-constitut}
\langle A_2u, (\phi_{\Omega}, \phi_\Sigma, \phi_{\Gamma_{\dyn}})\rangle_{\mathbb {L}^2} = \ft(u,\phi),
\end{equation}
for all test functions $\phi \in C_D^\infty(\Omega)$.

 If bulk diffusion degenerates towards $\Gamma_{\dyn}$ or $\Sigma$ as in \eqref{dist},
we rely on the weighted Sobolev embedding
$$W^{1,2}(\R^d,\dist(\cdot,S)^\gamma)\subset W^{\theta,q}(\R^d), \qquad 1-\frac{d+\gamma}{2}
 \geq \theta - \frac{d}{q}, \qquad q\geq 2,$$
which seems to be new in this explicit form and is deduced from the very 
general embedding results in \cite{HS} 
(see Proposition \ref{prop:Sob-weight} and \cite{AKS, Schumacher} for related results about traces of Muckenhoupt weighted spaces).
Here, $W^{\theta,q}(\R^d)$ denotes the usual Slobodetskii space.\\
It turns out that $-A_2$ generates an analytic $C_0$-semigroup $T_2(\cdot)$ of 
contractions on $\L^2$, see Proposition \ref{prop:AE1}.
This already yields the solvability of our realization of \eqref{e-parabol}--\eqref{e-initial} for external forces
$(f_{\Omega}, f_{\Gamma_{\dyn}}, f_\Sigma)$ in $L^2(J;\L^2)$ and initial data $u_0\in \L^2$. 
We emphasize that the components of the initial data need not be related, but that the semigroup regularizes to $u(t)\in \dom(\ft)$ for all
$t>0$. \\
In order to treat semilinear problems, $\L^2$-estimates of the solution will in 
general not be sufficient, due to the lack of embeddings for the fractional power domains of
$A_2$ into spaces of bounded functions. Thus, we first extend the definition of $A_2$ consistently 
to the whole $\L^p$-scale, $p\in [1,\infty]$. This is achieved by showing that $T_2(\cdot)$
is $\L^\infty$-contractive (see Proposition \ref{p:markovian}), which implies the existence
of a consistent contraction semigroup $T_p(\cdot)$ on $\L^p$ by interpolation and duality.
For $p\in(1,\infty)$, the negative generator $A_p$ of the analytic semigroup $T_p(\cdot)$ is then the desired 
consistent extension/restriction of $A_2$ to $\L^p$.
The analyticity of $T_p(\cdot)$ for $p\in (1,\infty)$ together with the contractivity of 
$T_p(\cdot)$ for $p\in [1,\infty]$ now allow us to apply a deep result from harmonic analysis 
due to \cite{Duong, KW, Lamberton, Merdy, Weis} (see also \cite[Proposition 2.2]{LX}) 
to conclude that $A_p$ admits a bounded holomorphic functional calculus and maximal Lebesgue 
regularity (see \cite{Dore, KuWe, Pruess} for  surveys on these topics). 

Hence, from an 
abstract point of view, the realization is as good as it can be, despite of the variety
of nonsmooth effects it takes into account. The precise formulation is given in 
Theorems \ref{thm:1} and \ref{thm:mr2}. 
Employing again that $A_p$ is given on a scalar $L^p$-space, we show that the 
multiplication with the inverse relaxation coefficient $\zeta^{-1}$ 
does not change the described properties. 
Finally, embeddings of the type 
\begin{equation}\label{emb}
\dom(A^\theta_p) \subset \mathbb{L}^\infty,
  \end{equation}
for $p > 2$ sufficiently large and $\theta$ sufficiently close to $1$ are obtained in 
Section \ref{sec:emb} from semigroup estimates and an integral formula for negative
fractional powers of $A_p$. We can quantify how the presence of surface diffusion 
may improve \eqref{emb}, whereas degeneracy in the bulk diffusion may clearly decrease 
the integrability exponent. It is an advantage of our unified framework that we can 
see how these effects may interact locally. In essence, we restrict our considerations to the linear case
in this paper, and refer e.g. to  \cite[Ch.~2]{henry}, \cite{luna} for results on how embeddings of type \eqref{emb} quantify the solvability of related semilinear problems.\\

This paper is organized as follows. We start in Section \ref{sec:Heu} with a heuristic of how our functional analytic setting is related to \eqref{e-parabol} -- \eqref{e-initial}.
In Section \ref{sec:SG}, we introduce tangent spaces and the surface gradient for Lipschitz
hypersurfaces in graph representation. 
In order to separate technical difficulties, in Section \ref{sec:nondeg} we consider the 
case of nondegenerate bulk diffusion only, while in Section \ref{sec:deg} we treat 
degenerate bulk diffusion. In Section \ref{sec:emb}, embeddings of fractional power domains into spaces of bounded 
functions are investigated. \smallskip

\textbf{Notation.} Generic positive constants are denoted by $C$ or $c$.
By $\mathcal L(\R^d)$ we designate the space of linear operators on $\R^d$, which we may 
identify with the set of $(d\times d)$-matrices via the canonical basis. The Euclidian scalar
 product of $x,y \in \R^d$ is denoted by $x \cdot y$ or $(x,y)$. For $p\in [1,\infty]$,
    the usual complex Lebesgue space is denoted by $L^p(\Omega)$.

\section{Heuristics}\label{sec:Heu}
Since the form method in \cite{AE} is very recent and presently not commonly known we give a detailed 
heuristics why the definition of the form $\mathfrak t$, together with the relation 
\eqref{e-constitut}
provides the adequate functional analytic setting for the equations \eqref{e-parabol} - 
\eqref{e-initial}. This is closely related to the classical arguments for weak formulations of boundary 
value problems, cf. for example 
\cite[Ch.~II.2]{ggz}. 
In this section, we make additional regularity assumptions. 
Let $\Omega$ be a smooth domain and let $\Sigma$ be extendible to a Lipschitz hypersurface $\Lambda = 
\overline{\Sigma} \cup (\Lambda\setminus \Sigma)$ which cuts $\Omega$ into two Lipschitz subdomains 
$\Omega = \Omega_+ \cup \Lambda \cup \Omega_-$. Let $\nu_\Sigma$ denote the outer normal vector field 
of $\Omega_+$ at all of  $\Lambda$. Assume that the equation 
\begin{equation}\label{Auf-eq}
	A_2u = f
	\end{equation}
	is satisfied in $\mathbb{L}^2$ and let $\phi \in C_D^\infty(\Omega)$ with the canoncical
 embedding $(\phi_\Omega, \phi_\Sigma, \phi_{\Gamma_\dyn}) \in \mathbb{L}^2$. Then by definition, 
\begin{equation}\label{f-eq}
	\langle f, \phi \rangle_{\mathbb{L}^2} = \int_\Omega f \overline{\phi} \,dx+ 
\int_\Sigma f_\Sigma \overline{\phi_\Sigma} \,d\mathcal H_{d-1}+ \int_{\Gamma_{\dyn}} f_{\Gamma_{\dyn}}
 \overline{\phi_{\Gamma_{\dyn}}} \,d\mathcal H_{d-1},
	\end{equation}
	and
	\begin{equation}\label{Auf-t-eq}
		 \langle A_2u, \phi \rangle_{\mathbb{L}^2} = \int_\Omega (\mu_\Omega \nabla u,
 \overline{\nabla \phi})\,dx + \int_{\Gamma_\dyn} \big (\mu_{\Gamma_\dyn}\nabla_{\Gamma_\dyn} u, 
		\overline{\nabla_{\Gamma_\dyn} \phi}\big) \,d\mathcal H_{d-1} 
		+ \int_{\Sigma} \big(\mu_\Sigma\nabla_\Sigma u, \overline{\nabla_\Sigma \phi}\big) 
		\,d\mathcal H_{d-1}.
	\end{equation}
Now we additionally assume that the restrictions $u_+$ and $u_-$ of $u$ to $\Omega_+$ and  $\Omega_-$
 satisfy $u_+ \in C^1(\overline{\Omega_+})$ and $u_- \in C^1(\overline{\Omega_-})$ and that on  
$\Omega \setminus \Sigma$, we have $u \in C^2(\Omega\setminus\Sigma)$. We note that 
$$
\int_\Omega (\mu_\Omega \nabla u, \overline{\nabla v})\,dx = \int_{\Omega_+} (\mu_\Omega \nabla u_+,
 \overline{\nabla \phi_+})\,dx + \int_{\Omega_-} (\mu_\Omega \nabla u_-, \overline{\nabla \phi_-})\,dx 
$$
and apply Gauss' Theorem to each of these terms to get
\begin{align*}
\int_\Omega (\mu_\Omega \nabla u, \overline{\nabla v})\,dx  & = \int_{\Omega_+} - \dv(\mu_\Omega \nabla u_+)
 \overline{\nabla \phi_+}\,dx + \int_{\Omega_-} - \dv(\mu_\Omega \nabla u_-) \overline{\phi_-}\,dx \\
& + \int_{\Gamma_N} (\nu \cdot \mu_\Omega \nabla u) \overline{\phi_{\Gamma_N}} \,d\mathcal H_{d-1} +
 \int_{\Gamma_{\dyn}} (\nu\cdot \mu_\Omega \nabla u) \overline{\phi_{\Gamma_d}} \,d\mathcal H_{d-1} \\
& + \int_{\Sigma} [\nu_\Sigma \cdot \mu_\Omega \nabla u] \overline{\phi_{\Sigma}} \,d\mathcal H_{d-1}
 + \int_{\Lambda \setminus \Sigma} [\nu_\Sigma \cdot \mu_\Omega \nabla u] 
\overline{\phi_{\Lambda\setminus\Sigma}} \,d\mathcal H_{d-1},
\end{align*}
where it follows from the regularity assumptions on $u$ that the last term vanishes. 
Additionally applying the manifold Gauss Theorem, cf. \cite{MMT} for a non-smooth version, 
to the last two integrals in \eqref{Auf-t-eq}, we derive the expression
\begin{align}\label{sum-eq}
\langle A_2u, \phi) \rangle_{\mathbb{L}^2}  & = \int_{\Omega} - \dv(\mu_\Omega \nabla u) \overline{\phi}\,dx 
+ \int_{\Gamma_N} (\nu \cdot \mu_\Omega \nabla u) \overline{\phi_{\Gamma_N}} \,d\mathcal H_{d-1} \\
 \nonumber
& + \int_{\Gamma_{\dyn}} (\nu\cdot \mu_\Omega \nabla u) \overline{\phi_{\Gamma_\dyn}} \,d\mathcal H_{d-1}
 + \int_{\Sigma} [\nu_\Sigma \cdot \mu_\Omega \nabla u] \overline{\phi_\Sigma} \,d\mathcal H_{d-1} \\
 \nonumber
& + \int_{\Gamma_\dyn} - \dv_{\Gamma_\dyn} (\mu_{\Gamma_\dyn}\nabla_{\Gamma_\dyn} u) 
\overline{\phi_{\Gamma_\dyn}} \,d\mathcal H_{d-1} 
		+ \int_{\Sigma} - \dv_\Sigma(\mu_\Sigma\nabla_\Sigma u) \overline{\phi_\Sigma}
\,d\mathcal H_{d-1} \\ \nonumber
		& +  \int_{\partial \Gamma_\dyn} (\nu_{\partial \Gamma_{\dyn}} \cdot 
\mu_{\Gamma_\dyn} \nabla_{\Gamma_\dyn} u_{\Gamma_\dyn}) 
		\overline{\phi_{\partial\Gamma_\dyn}} \,d\mathcal H_{d-2} + 
\int_{\partial \Sigma} (\nu_{\partial \Sigma}\cdot \mu_\Sigma\nabla_\Sigma u_\Sigma )
\overline{\phi_{\partial\Sigma}}\,d\mathcal H_{d-2} 
\end{align}
to be balanced with \eqref{f-eq}. \\
Choosing $\phi \in C_c^\infty (\Omega)$ yields $$f_\Omega = -\mathrm{div}(\mu_\Omega \nabla u)
 \in L^2(\Omega).$$ The Neumann and Dirichlet boundary conditions on $\Gamma_N$ and $\Gamma_D$
 follow, for example, as in \cite[Ch.~II.2]{ggz}, using that each Neumann part of the boundary
 of $\Omega$ satisfies an extension property. The remaining equalities $$f_{\Gamma_\dyn} = 
-\mathrm{div}_{\Gamma_\dyn}(\mu_{\Gamma_\dyn} \nabla_{\Gamma_\dyn} u_{\Gamma_{\dyn}}) + 
\nu \cdot \mu_{\Omega} \nabla u \in L^2(\Gamma_{\dyn})$$ and $$f_{\Sigma} = 
-\mathrm{div}_{\Sigma}(\mu_{\Sigma} \nabla_{\Sigma} u_{\Sigma}) + [\nu_\Sigma \cdot \mu_{\Omega} \nabla u]
 \in L^2(\Sigma)$$ are then identified accordingly. The last two terms in \eqref{sum-eq} 
require some more explanation. If $\partial\Gamma_{\dyn} \cup \partial\Sigma \subset 
\overline{\Omega} \setminus \Gamma_D$, we consider them to be enforcing (generalized) 
homogeneous Neumann boundary conditions on $\partial \Gamma_\dyn$ and $\partial \Sigma$. At points
 where $\partial \Gamma_\dyn$ or $\partial \Sigma$ and $\Gamma_D$ intersect, we assign 
homogeneous Dirichlet boundary conditions. In particular, in the definition of $C_D^\infty(\Omega)$,
 any subset of points in $\partial \Gamma_\dyn$ and $\partial \Sigma$ may be included to 
enforce these Dirichlet conditions. We did not include these conditions in equations \eqref{e-parabol} - 
\eqref{e-initial} to keep the presentation simple and because in general, our regularity 
assumptions on $\Gamma_\dyn$ and $\Sigma$ are insufficient to deduce them in the usual way. 


\section{The surface gradient on Lipschitz hypersurfaces}\label{sec:SG}
In order to define surface diffusion on $\Sigma$ and $\Gamma_{\dyn}$, in this section we introduce tangent spaces and the surface gradient for a Lipschitz hypersurface $\mathcal S$ in graph representation in an elementary way. 
The idea is that Lipschitz coordinates are differentiable almost everywhere, which allows us to give definitions in coordinates analogous to the smooth case. 
Hence for smooth $\SSS$ we automatically recover the standard notions, see \cite[Chapter VII]{AmEs} and \cite{Hebey, Jost} for basic accounts. 
For Lipschitz surfaces we also refer to \cite{EG, Grisvard, NS, Simon}.

\subsection{Lipschitz hypersurfaces} Let $\SSS\subset \R^{d}$ be a \emph{Lipschitz hypersurface in graph representation}. This means that for each $x\in \SSS$ 
there are  \emph{Lipschitz-graph coordinates} $(\g,U)$ and an  open neighbourhood $V$  of $x$ in $\R^d$ such that $U \subset \R^{d-1}$ is open and $\g:U \to \SSS\cap V$
is bijective and of the form 
$$\g(y) = Q\left ( \begin{matrix} y\\h(y)\end{matrix}\right) + x^*, \qquad y\in U,$$
where $Q\in \mathcal L(\R^d)$ is orthogonal, $x^*\in \R^d$ is a fixed vector and $h:U \to \R$ is Lipschitz continuous. For this and equivalent definitions we refer to \cite[Section 2]{NS}. We endow $\mathcal S$ with the Hausdorff measure $\mathcal H_{d-1}$. Employing that the topology of $\R^d$ has a countable basis, standard arguments show that there is an at most countable number of Lipschitz graph coordinates $(\g_\alpha,U_\alpha)$  such that $\SSS \subseteq \bigcup_\alpha \g_\alpha(U_\alpha)$, see the proof of \cite[Theorem 2.15]{NS}.

By Rademacher's theorem (see \cite[Theorem 3.1.2]{EG}), Lipschitz coordinates $\g$ are almost everywhere  differentiable on $U$ in the classical sense and one has $g\in W^{1,\infty}(U,\R^d)$, where
$$\g'(y) = Q \left ( \begin{matrix} \text{id}_{d-1} \\h'(y) \end{matrix}\right)\in \mathcal L(\R^{d-1},\R^d)$$
at points $y\in U$ where $\g$ is differentiable.
Observe that $\g'(y)$ is injective and has rank $d-1$. Hence the corresponding \emph{metric tensor} $\G:U\to \mathcal L(\R^{d-1})$, defined by
$$\G(y) = \g'(y)^T \g'(y) = \big((\partial_i \g(y), \partial_j \g(y))\big)_{ij},$$
is for almost all $y\in U$ symmetric and positive definite. With the usual abuse of notation we write $\G = (\g_{ij})_{ij}$, and $\G^{-1} = (\g^{ij})_{ij}$ for the pointwise inverse of $\G$.

We call Lipschitz-graph coordinates $\g$ \emph{regular} for $x\in \SSS$ if $\g$ is differentiable at $y = \g^{-1}(x)$.
If such regular coordinates exist, we call $x$ regular. 

\begin{lemma} Let $\SSS$ be a Lipschitz hypersurface in graph representation. Then $\mathcal H_{d-1}$-almost every point $x\in \SSS$ is regular.
\end{lemma}
\begin{proof} Let $N\subset \SSS$ be the set of points which are not regular. Take at most countable many coordinates $(\g_\alpha, U_\alpha)$ such that $\SSS \subseteq \bigcup_\alpha V_\alpha$ for $V_\alpha = \g_\alpha(U_\alpha)$. Then $\mathcal H_{d-1}(N) \leq \sum_\alpha \mathcal H_{d-1}( N \cap V_\alpha)$. Let further $N_\alpha \subset U_\alpha$ be the set of points where $\g_\alpha$ is not differentiable. Then $\mathcal H_{d-1}(N_\alpha) = 0$ by Rademacher's theorem. Using $N\cap V_\alpha \subseteq \g_\alpha(N_\alpha)$ and \cite[Theorem 2.4.1/1]{EG}, for each $\alpha$ we obtain 
$$\mathcal H_{d-1}( N \cap V_\alpha) \leq \mathcal H_{d-1}(\g_\alpha(N_\alpha)) \leq \text{Lip}(\g_\alpha)^{d-1} \mathcal H_{d-1}(N_\alpha) =0,$$
where $\text{Lip}(\g_\alpha)$ is the Lipschitz constant of $\g_\alpha$. This shows $\mathcal H_{d-1}(N) = 0$.
\end{proof}

As another preparation we consider the properties of \emph{transition maps}.

\begin{lemma} \label{lem:tech-S} Let $(\emph{\g}_\alpha,U_\alpha)$ and $(\emph{\g}_\beta,U_\beta)$ be Lipschitz-graph coordinates for $\SSS$ which are both regular for $x\in \SSS$. Set $y_\alpha = \emph{\g}_\alpha^{-1}(x)\in U_\alpha$ and $y_\beta = \emph{\g}_\beta^{-1}(x)\in U_\beta$. Then the following assertions hold true.
\begin{itemize}
\item[\text{(a)}] The transition map $\emph{\g}_\beta^{-1}\circ \emph{\g}_\alpha$ is differentiable at $y_\alpha$. The derivative $(\emph{\g}_\beta^{-1}\circ \emph{\g}_\alpha)'(y_\alpha)\in \mathcal L(\R^{d-1})$ is invertible with inverse $(\emph{\g}_\alpha^{-1}\circ \emph{\g}_\beta)'(y_\beta).$
\item[\text{(b)}] The derivatives $\emph{\g}_\alpha'(y_\alpha)$ and $\emph{\g}_\beta'(y_\beta)$ have the same images in $\R^d$. We have $v = \emph{\g}_\alpha'(y_\alpha)\xi_\alpha$ for $\xi_\alpha\in \R^{d-1}$ if and only if $v = \emph{\g}_\beta'(y_\beta)\xi_\beta$ for $\xi_\beta =(\emph{\g}_\beta^{-1}\circ \emph{\g}_\alpha)'(y_\alpha)\xi_\alpha$.
\item[\text{(c)}] For the metric tensors $\emph{\G}_\alpha$ and $\emph{\G}_\beta$ corresponding to $\emph{\g}_\alpha$ and $\emph{\g}_\beta$ we have
$$\emph{\G}_\alpha(y_\alpha) = (\emph{\g}_\beta^{-1}\circ \emph{\g}_\alpha)'(y_\alpha)^T \emph{\G}_\beta(y_\beta) 
(\emph{\g}_\beta^{-1}\circ \emph{\g}_\alpha)'(y_\alpha).$$
\end{itemize}
\end{lemma}
\begin{proof} We write $\Phi = \g_\beta^{-1}\circ \g_\alpha$ for the transition map. Observe that $\Phi$ is a homeomorphism on a neighbourhood of $y_\alpha$ with inverse $\Phi^{-1} = \g_\alpha^{-1}\circ \g_\beta$.

(a) The form of $\g_\beta$ shows that $\Phi(y)$ is given by the first $d-1$ entries of $Q_\beta^T(\g_\alpha(y)-x_\beta^*)$. Hence $\Phi$ is differentiable at $y_\alpha$. In the same way we obtain the differentiability of $\Phi^{-1}$ at $y_\beta$. Therefore $\Phi'(y_\alpha)$ is invertible with inverse as asserted.

(b) This follows from $\g_\alpha'(y_\alpha) = \g_\beta'(y_\beta) \Phi'(y_\alpha)$ and the invertibility of $\Phi'(y_\alpha)$.

(c) We can repeat the short argument from \cite[Section 1.4]{Jost}. For arbitrary $\xi_\alpha,\eta_\alpha\in \R^{d-1}$ we use (b) to obtain
\begin{align*}
(\G_\alpha(y_\alpha) \xi_\alpha,\eta_\alpha)&\,= \big(\g'_\alpha(y_\alpha) \xi_\alpha, \g'_\alpha(y_\alpha) \eta_\alpha\big)\\
&\, = \big(\g'_\beta(y_\beta) \Phi'(y_\alpha) \xi_\alpha, \g'_\beta(y_\beta) \Phi'(y_\alpha) \eta_\alpha\big)\\
&\, = (\Phi'(y_\alpha)^T\G_\beta(y_\beta)\Phi'(y_\alpha) \xi_\alpha,\eta_\alpha).
\end{align*}
This implies the asserted formula.
\end{proof}

\subsection{Tangent space and surface gradient} Now we can introduce the following notions. 

\begin{definition} \label{def:tang} Let $\SSS$ be a Lipschitz hypersurface in graph representation.
\begin{itemize}
 \item[\text{(a)}] Let $x\in \SSS$ be regular with Lipschitz graph coordinates $(\g,U)$. The \emph{tangent space} at $x$ is
 $$T_x\SSS  = \big \{v\in \R^d\,:\,\text{there is } \xi\in \R^{d-1} \text{ with } v = \g'(\g^{-1}(x))\xi\big \}.$$
 \item[\text{(b)}] A function $u\in L_{\text{loc}}^1(\SSS)$ is called \emph{weakly differentiable}, 
 if for all Lipschitz graph coordinates $(\g,U)$ for $\SSS$ the function $u\circ \g$ is weakly differentiable on $U \subset \R^{d-1}$. 
 \item[\text{(c)}] Let $u\in L_{\text{loc}}^1(\SSS)$ be weakly differentiable. Then for a regular point $x \in \SSS$ the \emph{surface gradient} $\nabla_{\!\SSS}u(x)\in T_x\SSS$ is given by
\begin{equation*}
 \nabla_{\!\SSS} u(x) = \g'(y) \G^{-1}(y) \nabla (u\circ \g)(y) = \sum_{i,j=1}^{d-1}\g^{ij}(y)\partial_j(u\circ \g)(y)\partial_i \g(y),
\end{equation*}
where $(\g,U)$ are arbitrary regular Lipschitz graph coordinates for $x$ and $y= \g^{-1}(x)$. 
\end{itemize}
\end{definition}

These notions coincide with the usual ones if $\SSS$ is smooth, see, e.g., \cite[Remark VII.10.11]{AmEs} for the representation of the surface gradient in coordinates. As in the smooth case one shows that these notions are well-defined.

\begin{lemma} At a regular point $x\in \SSS$, the tangent space as well as the surface gradient of a weakly differentiable function are independent of the chosen regular graph coordinates.
\end{lemma}
\begin{proof} The assertion for the tangent space follows from Lemma \ref{lem:tech-S}(b). For the surface gradient we let $\g_\alpha$ and $\g_\beta$ be regular for $x$, set $y_\alpha = \g_\alpha^{-1}(x)$, $y_\beta = \g_\beta^{-1}(x)$ and
$$v_\alpha = \g_\alpha'(y_\alpha)\G_\alpha^{-1}(y_\alpha) \nabla(u\circ \g_\alpha)(y_\alpha), \qquad v_\beta = \g_\beta'(y_\beta)\G_\beta^{-1}(y_\beta) \nabla(u\circ \g_\beta)(y_\beta).$$
As above we write $\Phi = \g^{-1}_\beta \circ \g_\alpha$ for the transition map. By Lemma \ref{lem:tech-S}(b) we have $v_\alpha = v_\beta$ if and only if
$$\G_\beta^{-1}(y_\beta) \nabla(u\circ \g_\beta)(y_\beta) = \Phi'(y_\alpha)\G_\alpha^{-1}(y_\alpha) \nabla(u\circ \g_\alpha)(y_\alpha).$$
But this is a consequence of the identities
$$\nabla(u\circ \g_\alpha)(y_\alpha) = \Phi'(y_\alpha)^T \nabla (u\circ \g_\beta)(y_\beta), \qquad \G_\beta^{-1}(y_\beta) = \Phi'(y_\alpha)\G_\alpha^{-1}(y_\alpha) \Phi'(y_\alpha)^T,$$
where the latter follows from Lemma \ref{lem:tech-S}(c).
\end{proof}

\section{Non-degenerate bulk diffusion}\label{sec:nondeg}
In this section we consider \eqref{e-parabol}--\eqref{e-initial} with a uniformly elliptic 
diffusion coefficient $\mu_\Omega$ in the bulk. The case when $\mu_\Omega$ degenerates towards
 a compact Lipschitz surface is investigated in the next section.
\subsection{Assumptions on the geometry and the coefficients} 
In case $1\leq k\leq d-1$, we say that the set $\SSS$ is a $(d-k)$-dimensional \emph{Lipschitz 
submanifold} if for all $x\in \SSS$ there is an open neighbourhood $V$ of $x$ in $\R^d$ and a 
bi-Lipschitz mapping $\varphi$ from $V$ to $\R^d$ such that $\varphi(\SSS\cap V) =
]0,1[^{d-k} \times \{0_{\R^k}\}$. By a compact $0$-dimensional Lipschitz submanifold $\SSS$ we mean a finite union of points. 
Throughout the paper, we impose the following.

\begin{assu}\label{assu:geo}
\begin{itemize}
 \item[\text{(a)}] $\Omega\subset \R^d$ is a bounded domain, $d\geq 2$.
 \item[\text{(b)}] $\Gamma_{\dyn} \subset \partial\Omega$ and $\Sigma\subset \Omega$ are  Lipschitz 
hypersurfaces in graph representation. They are endowed with the $(d-1)$-dimensional Hausdorff 
measure $\mathcal H_{d-1}$.
 \item[\text{(c)}] $\Gamma_N$ is a $(d-1)$-dimensional Lipschitz submanifold of $\mathbb{R}^d$.
 \item[\text{(d)}] Additionally, the closures $\overline{\Gamma_N}$, $\overline{\Gamma_{\dyn}}$ and 
$\overline{\Sigma}$ are contained in $(d-1)$-dimensional Lipschitz submanifolds, respectively.
 \end{itemize}
\end{assu}

	We emphasize that for the Dirichlet part $\Gamma_D$, there are only assumptions in a 
neighbourhood of points where $\Gamma_D$ meets $\Gamma_N$ or $\Gamma_\dyn$. In particular, in
 the pure Dirichlet case $\Gamma_D =  \partial\Omega$ there are no assumptions on the boundary.
 It is not excluded that one or more of the sets $\Gamma_D$, $\Gamma_N$, $\Gamma_{\dyn}$ 
or $\Sigma$ are empty.

\begin{assu}\label{assu:coeff}
	\begin{itemize}
\item[\text{(a)}] The coefficient $\mu_\Omega: \Omega \to \mathcal L(\R^d)$ is measurable, 
bounded and there is a constant $\mu_\Omega^* > 0$ such that
 $$\big(\mu_\Omega(x) \xi, \xi\big) \geq \mu_\Omega^* |\xi|^2, \qquad  x\in \Omega, 
\qquad \xi \in \R^d.$$
  \item[\text{(b)}] Let $\SSS$ be either $\Gamma_{\dyn}$ or $\Sigma$. 
  Then $\mu_\SSS: \SSS\to \mathcal L(\R^d)$ is measurable, and there are a measurable,
 bounded, nonnegative
  function $\mu_\SSS^*:\SSS\to \R$ and constants $c_1,c_2 > 0$ such that
$$\big(\mu_\SSS(x)\xi,\xi\big) \geq c_1\mu_\SSS^*(x)|\xi|^2, \qquad \|\mu_\SSS(x)\|_{\mathcal L(\R^d)}
 \leq c_2\mu_\SSS^*(x), 
  \qquad x\in \SSS, \quad \xi \in T_x\SSS.$$
\item[\text{(c)}] The relaxation coefficient $\zeta: \Omega \cup \Gamma_{\dyn}\to \R$ is measurable,
 bounded and there is a constant $c > 0$ such that $\zeta(x)\geq c$ for all 
$x\in  \Omega \cup \Gamma_{\dyn} \cup \Sigma$.
\end{itemize}
\end{assu}

The functions $\mu_{\Gamma_{\dyn}}^*$ and $\mu_\Sigma^*$ describe where diffusion takes place on 
$\Gamma_{\dyn}$ and $\Sigma$, and where diffusion degenerates. There are no restrictions on the 
support of these functions.
An example we have in mind is $\mu_\SSS^*(x) = \dist(x,M)^\gamma$ for a subset $M\subset \SSS$ 
and $\gamma > 0$, which indicates that diffusion degenerates towards $M$ and is impossible 
along and across $M$.

\begin{rem}
The above assumptions cover a large class of nonsmooth scenarios.
However, our realization of \eqref{e-parabol}--\eqref{e-initial} developed below also works under more general conditions. For instance, the interface $\Sigma$ must only be a Lipschitz hypersurface in graph representation in a neighbourhood of the support 
of $\mu_\Sigma^*$. Away from the support, as in \cite{EMR} it suffices that $\Sigma$ is a $(d-1)$-set (see \cite[Section VII.1.1]{JW}). To avoid too many technical difficulties we do not take these issues into account.
\end{rem}

\subsection{The realization on $\mathbb L^2$} We construct the operator $A_2$ which yields a realization of the 
elliptic part of \eqref{e-parabol}--\eqref{e-initial} on a suitable $L^2$-space, cf. Section \ref{sec:Heu}.
 The approach based on sesquilinear forms is similar to the one used in \cite{EMR}. 
     
For $p \in (1,\infty)$ we denote by $W^{1,p}(\Omega)$ the usual complex Sobolev space over $\Omega$. We further 
define $W_\Gam^{1,p}(\Omega)$ as the closure in $W^{1,p}(\Omega)$ of 
$$C_\Gam^{\infty}(\Omega)= \big\{ u|_{\Omega} : u \in C_c^{\infty}({\R}^d), \; (\supp u )\cap \Gamma_D
  = \emptyset \big\}.$$
Roughly speaking, elements of $W_\Gam^{1,p}(\Omega)$ vanish on the Dirichlet part $\Gamma_D$ of $\partial \Omega$.

Let $\tr_{\Gamma_\dyn}$ and $\tr_\Sigma$ be the trace operators for $\Gamma_{\dyn}$ and $\Sigma$. Then \cite[Proposition 2.8]{EMR} implies the continuity of 
\begin{equation}\label{tr-cont}
 \tr_{\Gamma_\dyn}: W_\Gam^{1,2}(\Omega)\to L^2({\Gamma_\dyn}), \qquad \tr_\Sigma: W_\Gam^{1,2}(\Omega)\to L^2(\Sigma).
\end{equation}
As in the Introduction and Heuristics Sections, we use the notation $u_{\Gamma_\dyn} = \tr_{\Gamma_\dyn} u$ and $u_\Sigma = \tr_\Sigma u$ for the traces, and sometimes write only $u$ for $u_{\Gamma_\dyn}$ or $u_\Sigma$. 

\begin{definition} \label{def:domt}
\begin{itemize}
\item[\text{(a)}] On $C_\Gam^\infty(\Omega)$ we introduce the scalar product $(\cdot,\cdot)_{\dom(\mathfrak t)}$ by
$$(u,v)_{\dom(\mathfrak t)} = (u,v)_{W^{1,2}(\Omega)} + \int_{\Gamma_\dyn} \big(\nabla_{\Gamma_\dyn}  u, \overline{\nabla_{\Gamma_\dyn} v}\big) \,\mu_{\Gamma}^*\, d\mathcal H_{d-1} + \int_{\Sigma} \big(\nabla_\Sigma u, \overline{\nabla_\Sigma v}\big) \,\mu_{\Sigma}^* \,d\mathcal H_{d-1},$$ 
where $(\cdot,\cdot)_{W^{1,2}(\Omega)}$ is the usual scalar product on $W^{1,2}(\Omega)$. The corresponding Hilbert norm is denoted by $\|\cdot\|_{\dom(\mathfrak t)}$. 
\item[\text{(b)}] The Hilbert space $\dom(\mathfrak t)$ is defined by
$$\dom(\mathfrak t) = \text{completion of $C_\Gam^\infty(\Omega)$ with respect to $\|\cdot\|_{\dom(\mathfrak t)}$}.$$
\item[\text{(c)}] For $p\in [1,\infty]$ we set
$\L^p = L^p\big( (\Omega\setminus\Sigma) \cup\Gamma_{\dyn}\cup\Sigma, (dx+ \mathcal H_{d-1})\big).$
\item[\text{(d)}] The map $\mathfrak J: \dom(\mathfrak t)\to \mathbb L^2$ is given by
$\mathfrak J(u) = (u, u_{\Gamma_{\dyn}}, u_\Sigma).$
\end{itemize}
\end{definition}

For $\SSS\in \{\Gamma_\dyn, \Sigma\}$ we will also write $$\|f\|_{L^2(\SSS,\mu_{\SSS}^*)}^2 = \int_\SSS |f|^2 \, \mu_{\SSS}^*\,d\mathcal H_{d-1},$$
such that the Hilbert norm may be expressed as
\begin{equation}\label{eq:hilbert-norm}
 \|u\|_{\dom(\ft)}^2 = \|u\|_{W^{1,2}(\Omega)}^2 + \|\nabla_{\Gamma_{\dyn}} u\|_{L^2(\Gamma_{\dyn}, \mu_{\Gamma_{\dyn}}^*)}^2 + \|\nabla_{\Sigma} u\|_{L^2(\Sigma, \mu_{\Sigma}^*)}^2.
\end{equation}
In view of $\dom(\mathfrak t)\subseteq W_D^{1,2}(\Omega)$ and the continuity of the traces \eqref{tr-cont}, the map $\J$ is indeed well-defined. 
The space $\L^p$ can be identified as
$$\L^p = L^p(\Omega\setminus \Sigma)\oplus L^p(\Gamma_{\dyn})\oplus L^p(\Sigma).$$

\begin{rem}
The space $\dom(\ft)$ includes an implicit definition of a weak surface gradient, 
    even if $\mu_{\Gamma_\dyn}^*$, $\mu_\Sigma^*$ are only non-negative, as
the operator 
\[
\nabla_\Sigma: \{\psi|_\Sigma: \psi \in C_D^\infty(\Omega)\}  \to L^2(\mu^*_\Sigma, \Sigma)
\]
 continuously extends to $\dom(\ft)$ by density
(analogously for $\Gamma_\dyn$). This implies our concept of
degenerate diffusion on $\Sigma$ and $\Gamma_\dyn$. The regularity of elements $u$ of $\dom(\ft)$ on $\Gamma_{\dyn}$ and $\Sigma$ is determined by the supports of 
$\mu_{\Gamma_\dyn}^*$ and $\mu_\Sigma^*$. On subsets where these are strictly positive, 
$u_{\Gamma_\dyn}$ and $u_\Sigma$ have square integrable weak surface gradients in the sense of Section \ref{sec:SG}.
\end{rem}

 The operator $A_2$ will be derived from the sesquilinear form
$$\mathfrak t(u,v) = \int_\Omega (\mu_\Omega \nabla u, \overline{\nabla v})\,dx + \int_{\Gamma_\dyn} \big (\mu_{\Gamma_\dyn}\nabla_{\Gamma_\dyn} u, \overline{\nabla_{\Gamma_\dyn} v}\big) \,d\mathcal H_{d-1} + \int_{\Sigma} \big(\mu_\Sigma\nabla_\Sigma u, \overline{\nabla_\Sigma v}\big) \,d\mathcal H_{d-1},$$
which is originally defined for $u,v \in C_D^\infty(\Omega)$.

\begin{lemma}\label{lem:t-J} The form $\mathfrak t$ extends continuously to a sesquilinear form on $\dom(\mathfrak t)$. It  is $\mathfrak J$-elliptic, i.e., there is $c >0$ such that
$$\Re \mathfrak t(u,u) + \|\mathfrak J u\|_{\mathbb L^2}^2 \geq c \|u\|_{\dom(\mathfrak t)}^2, \qquad u\in \dom(\mathfrak t).$$
Moreover, the map $\mathfrak J: \dom(\mathfrak t)\to \mathbb L^2$ has dense range and is continuous and compact.
\end{lemma}
\begin{proof} The continuity and the compactness of $\mathfrak J$ follow from $\dom(\mathfrak t)\subseteq W_D^{1,2}(\Omega)$ and \cite[Lemma 2.10]{EMR}. The proof in \cite{EMR} also shows that $\mathfrak J C_D^\infty(\Omega)$ is dense in $\L^2$, hence $\mathfrak J \dom(\mathfrak t)$ is dense since $C_D^\infty(\Omega)\subset \dom(\mathfrak t)$.

It is clear that $\ft:C_D^\infty(\Omega)\times C_D^\infty(\Omega) \to \C$ is sesquilinear. Given $u,v\in C_D^\infty(\Omega)$ we use the assumption $\|\mu_\SSS(x)\|_{\mathcal L(\R^d)} \leq c_2\mu_\SSS^*(x)$ for $\SSS\in \{\Gamma_{\dyn}, \Sigma\}$, H\"older's inequality and \eqref{eq:hilbert-norm} to estimate
\begin{align*}
 |\mathfrak t(u,v)| &\, \leq \|\mu_\Omega\|_\infty \|\nabla u\|_{L^2(\Omega)} \|\nabla v\|_{L^2(\Omega)}\\
&\, \qquad  +c_2\|\nabla_{\Gamma_\dyn} u\|_{L^2(\Gamma_\dyn,\mu_{\Gamma_\dyn}^*)} \|\nabla_{\Gamma_\dyn} v\|_{L^2(\Gamma_\dyn, \mu_{\Gamma_\dyn}^*)}+ c_2\|\nabla_{\Sigma} u\|_{L^2(\Sigma,\mu_\Sigma^*)} \|\nabla_\Sigma v\|_{L^2(\Sigma,\mu_\Sigma^* )} \\
&\, \leq C \|u\|_{\dom(\mathfrak t)} \|v\|_{\dom(\mathfrak t)} .
\end{align*}
Hence $\mathfrak t$ extends continuously to a sesquilinear form on $\dom(\ft)$.  To show its $\mathfrak J$-ellipticity, for $u\in C_D^\infty(\Omega)$ we use the assumption $(\mu_\SSS \xi,\xi) \geq c_1\mu_\SSS^*|\xi|^2$ for $\SSS\in \{\Gamma_{\dyn}, \Sigma\}$ to get
\begin{align*}
 \Re \mathfrak t(u,u) + \|\mathfrak J u\|_{\mathbb L^2}^2 &\, \geq \mu_\Omega^* \|\nabla u\|_{L^2(\Omega)}^2 + c_1\|\nabla_{\Gamma_\dyn} u\|_{L^2(\Gamma_{\dyn}, \mu_{\Gamma_{\dyn}}^*)}^2 + c_1\|\nabla_\Sigma u\|_{L^2(\Sigma, \mu_{\Sigma}^*)}^2 + \|u\|_{L^2(\Omega)}^2\\
&\, \geq c \|u\|_{\dom(\mathfrak t)}^2.
\end{align*}
This inequality carries over to all $u\in \dom(\mathfrak t)$ by density and the continuity of $\J$.
\end{proof}

Now the operator $A_2$ can be derived from $\mathfrak t$ as follows.

\begin{proposition} \label{prop:AE1} There is a closed, densely defined operator $A_2$ on $\mathbb L^2$ associated to the form $\mathfrak t$. For $\varphi,\psi \in \mathbb L^2$ we have $ \varphi\in \dom(A_2)$ and $A_2\varphi = \psi$ if and only if there is
$u\in \dom(\mathfrak t)$ such that $\varphi = \mathfrak J u$ and
$$(\psi, \mathfrak J v)_{\mathbb L^2} = \mathfrak t(u,v)\qquad \text{for all }v\in \dom(\mathfrak t).$$
The operator $-A_2$ generates an analytic $C_0$-semigroup 
$$T_2(\cdot)= (T_2(t))_{t\geq 0}$$ 
of contractions on $\mathbb L^2$. Furthermore, $A_2$ has compact resolvent. \end{proposition}
\begin{proof}
All assertions except the contraction property are a consequence of Lemma \ref{lem:t-J} and the general results of \cite[Theorem 2.1, Lemma 2.7]{AE}. For the contractivity we observe that for $\varphi\in \dom(A_2)$ with $\varphi = \mathfrak J u$ for $u\in \dom(\mathfrak t)$ we have $\Re(A_2\varphi, \varphi) = \Re \mathfrak t(u,u) \geq 0$. Hence the vertex of $A_2$ is zero and the contractivity of the semigroup follows from \cite[Theorem IX.1.24]{Kato}.
\end{proof}


\subsection{Properties of $A_2$ and extension to $\mathbb L^p$} \label{sec:extension} The key to the extension of $A_2$ to all $\mathbb L^p$-spaces is the $\mathbb L^\infty$-contractivity of the semigroup $T_2(\cdot)$. 
For the contractivity we will employ that $A_2$ is associated to the form $\ft$.
In this situation, suitable invariance criteria for closed convex sets are available.

By $\mathbb L^2_{\R}$ we denote the subspace of real-valued elements of $\mathbb L^2$.

\begin{proposition}\label{p:markovian} The semigroup $T_2(\cdot)$ generated by $-A_2$ leaves $\mathbb L^2_{\R}$ invariant, it is $\L^\infty$-contractive and positive.\end{proposition}
\begin{proof} The set $\mathbb L^2_{\R}$ is closed and convex, and $\varphi\mapsto \Re \varphi$ is the orthogonal projection onto $\L_\R^2$. For $u\in C_D^{\infty}(\Omega)$ we have $\Re\ft(u,u-\Re u) \geq 0$, and this inequality carries over to all $u\in \dom(\ft)$ by density. Hence each $T_2(t)$ leaves $\L^2_\R$ invariant by \cite[Proposition 2.9(iii)]{AE}.

For the $\L^\infty$-contractivity and the positivity, as in \cite[Prop. 2.16]{EMR} it suffices to show that $T_2(\cdot)$
leaves the closed and convex set $\mathcal C = \{\varphi \in \L^2_\R: \varphi \leq \mathds 1\}$ invariant. 
Again, we apply a criterion from \cite{AE}, on a dense subset of $\dom(\ft)$.

For a real-valued function $u$ we define $u\wedge \mathds 1$ by $(u\wedge \mathds 1)(x) = \min(u(x),1).$
The ortho\-gonal projection $P$ of $\L^2$ onto $\mathcal C$ is given by $P\varphi = (\Re \varphi)\wedge \1$. 
Moreover, for $u\in C_D^\infty(\Omega)$ one has $P\J u = \J((\Re u)\wedge \1)$ and 
$$\Re\mathfrak t((\Re u)\wedge \mathds 1,u-(\Re u)\wedge \mathds 1) = 0.$$
Hence, \cite[Proposition 2.9(iv)]{AE} yields the invariance of $\mathcal C$. 
\end{proof}
Now standard interpolation and duality arguments together with \cite[Proposition 3.12]{Ou} allow to extend $T_2(\cdot)$ to the entire $\L^p$-scale as follows.

\begin{proposition}\label{p:contraction} For all $p\in [1,\infty]$ the semigroup $T_2(\cdot)$ generated by $-A_2$ extends consistently to a contraction semigroup $T_p(\cdot)$ on $\mathbb L^p$, which is strongly continuous for $p\in [1,\infty)$ and analytic for $p\in (1,\infty)$.\end{proposition}
We define
$$\text{$A_p$ is the negative generator of $T_p(\cdot)$.}$$
Then $A_p$ coincides with $A_2$ on $\dom(A_p)\cap \dom(A_2)$. Let the relaxation coefficient $\zeta\in \L^\infty$ be as in Assumption \ref{assu:geo}. Rescaling in measure as in the proof of \cite[Theorem 2.21]{EMR} and using \cite[Proposition 2.20]{EMR}, we obtain that the operators $-\zeta^{-1}A_p$ generate consistent contractive semigroups on (the rescaled) $\L^p$ for $p\in [1,\infty]$, which are analytic for $p\in (1,\infty)$.

We can thus apply \cite[Proposition 2.2]{LX} to obtain the main result of this section.

\begin{theorem}\label{thm:1} For each $p\in (1,\infty)$ the operator $\zeta^{-1} A_p$ with domain $\dom(A_p)$ admits a bounded holomorphic functional calculus on $\mathbb L^p$, with angle strictly smaller than $\frac{\pi}{2}$.  As a consequence, $\zeta^{-1} A_p$ enjoys maximal parabolic $L^s$-regularity for all $s\in (1,\infty)$, and $-\zeta^{-1} A_p$ generates an analytic $C_0$-semigroup on $\L^p$. Furthermore, the fractional power domains are given by complex interpolation, i.e.,
$$\dom(A_p^\theta) = [\mathbb L^p, \dom (A_p)]_\theta,\qquad \theta\in [0,1],$$ 
and the resolvent of $\zeta^{-1} A_p$ is compact. \end{theorem}



\section{Degenerate bulk diffusion}\label{sec:deg}
In this section we generalize the above setting and allow for degeneracies in the bulk diffusion coefficient $\mu_\Omega$. 
Of special interest is the case when the degeneracy takes place at the dynamic boundary part $\Gamma_\dyn$ or 
the dynamic interface $\Sigma$. In this case the continuity of the map $\J:\dom(\ft)\to \L^2$, which is crucial for the approach used in the last section, depends on the degeneracy of the bulk diffusion.

Throughout we keep Assumption \ref{assu:geo}, but we replace the uniform ellipticity of $\mu_\Omega$ by the assumption that there are constants $c_1,c_2 > 0$ such that
\begin{equation}\label{e:muOmega}\big( \mu_\Omega(x)\xi,\xi\big) \geq c_1\mu_\Omega^*(x)|\xi|^2, \qquad \|\mu_\Omega(x)\|_{\mathcal L(\R^d)} \leq c_2 \mu_\Omega^*(x), \qquad x \in \Omega,\quad \xi\in \R^d,
\end{equation}
where $$\mu_\Omega^*(x) = \text{dist}(x,S)^\gamma$$
for a compact $(d-k)$-dimensional Lipschitz submanifold $S\subset \overline{\Omega}$, $1\leq k \leq d$, and $0 \leq \gamma <k$ for the distance exponent. We refer to  Section \ref{sec:nondeg} for a definition of a Lipschitz submanifold.

Observe that for $\gamma = 0$ we are in the nondegenerate situation of the previous section. 
We must distinguish the two cases 
\begin{align*}\text{Case } (\mathrm{A}): & \text{  $\mu_\Omega$ degenerates at a distance from 
the dynamics surfaces only, } S \cap (\overline{\Gamma_\dyn \cup \Sigma}) = \emptyset, \\ 
\text{Case } (\mathrm{B}): & \text{  $\mu_\Omega$ degenerates directly at the dynamics surfaces, } 
S \cap (\overline{\Gamma_\dyn \cup \Sigma}) \neq \emptyset, 
  \end{align*}
see the figure below.

\begin{figure}[h] 
\centering
\includegraphics[width=.3\textwidth]{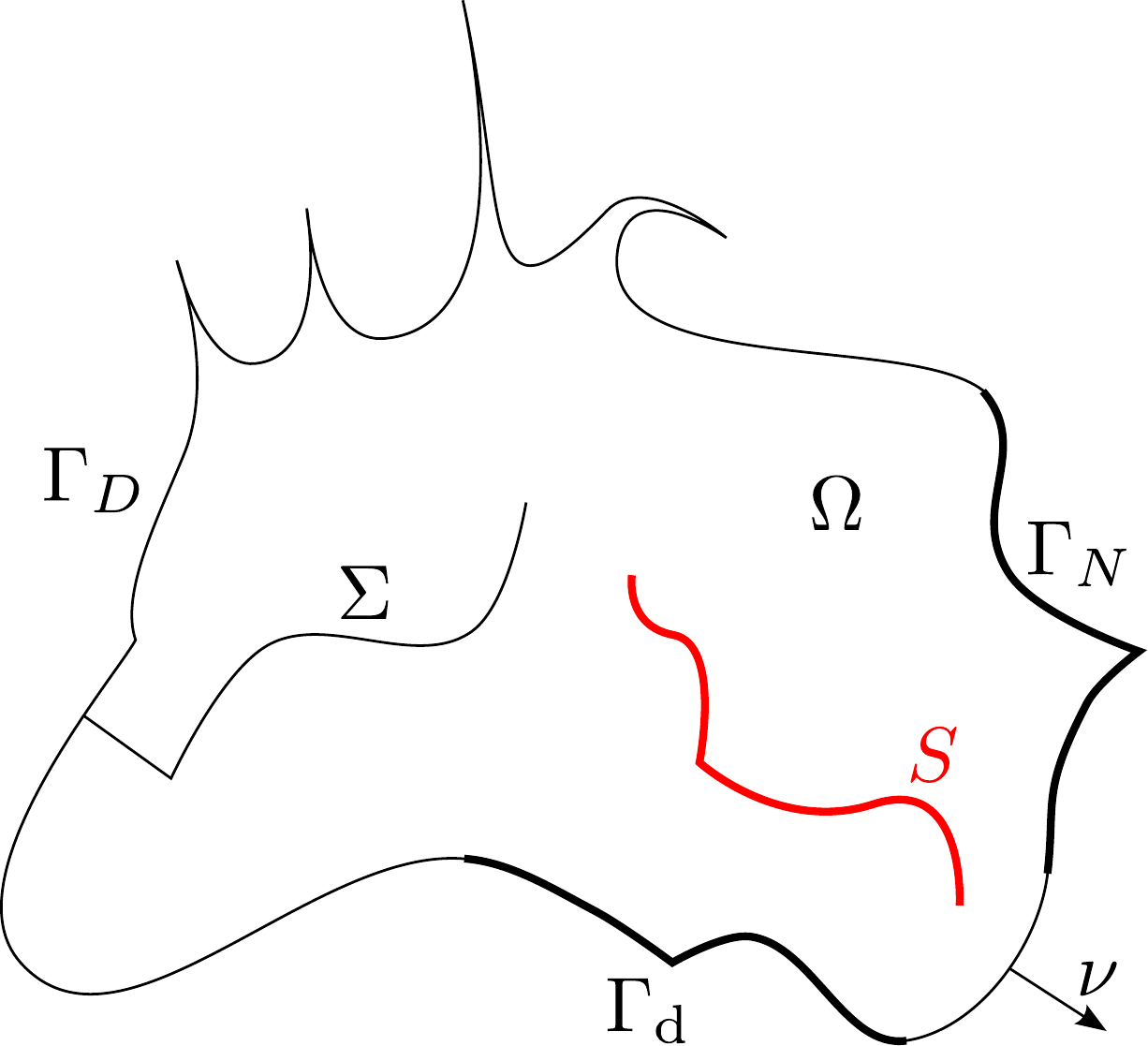}
\qquad \qquad \qquad                 
        ~ 
\includegraphics[width=.3\textwidth]{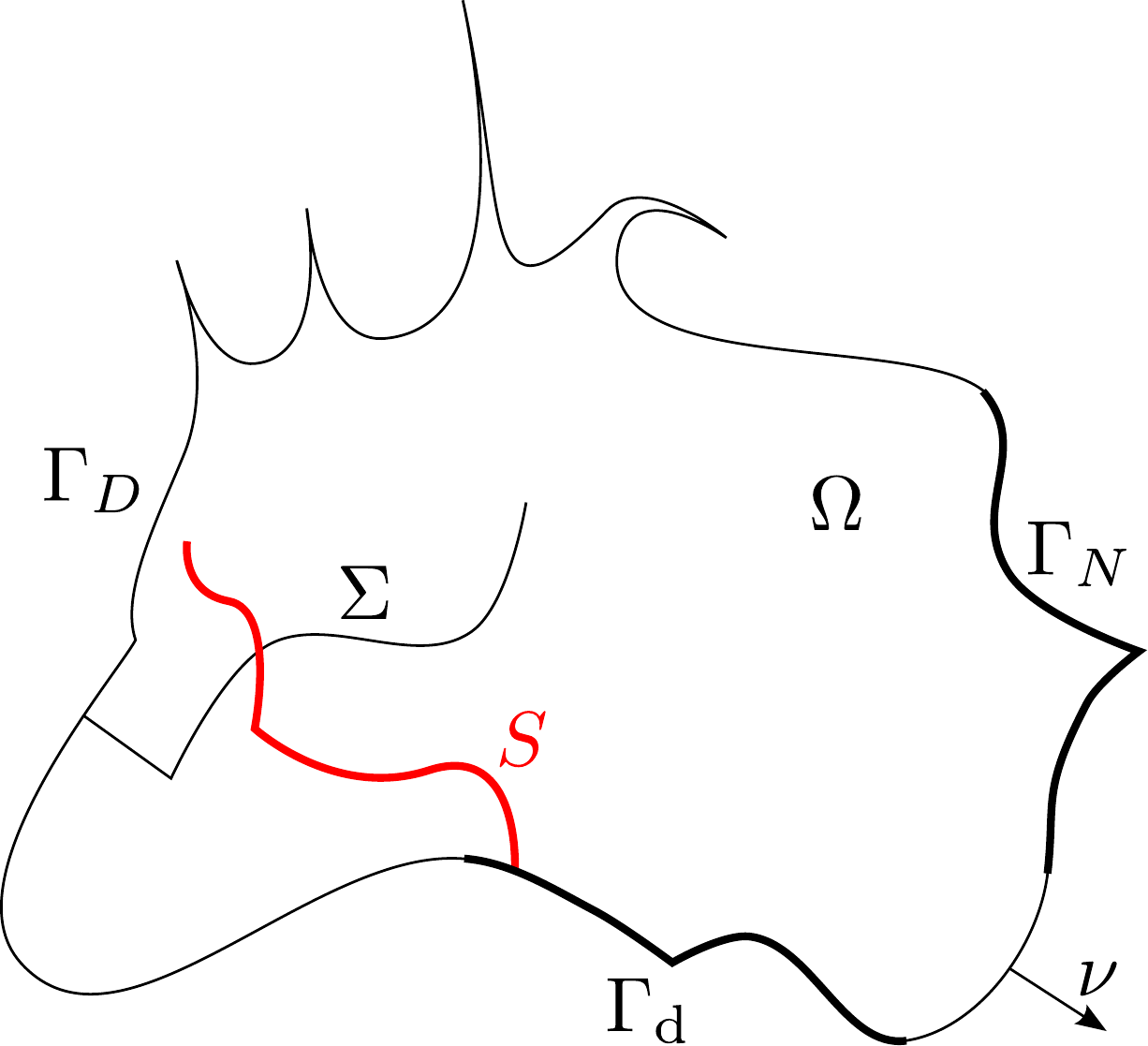}
                                       ~ 
       \caption{\qquad \quad Case (A) \quad \qquad \qquad \qquad \qquad \qquad \quad Case (B) \qquad \quad}
\end{figure} \label{fig:AB}


\subsection{Weighted function spaces} In order to incorporate the degeneracy of $\mu_\Omega$ into the domain of the sesquilinear form $\mathfrak t$ we introduce weighted function spaces.

We define $W^{1,2}_D(\Omega, \mu_\Omega^*)$ as the closure of $C_D^\infty(\Omega)$ with respect to the norm $\|\cdot\|_{W^{1,2}_D(\Omega, \mu_\Omega^*)}$, which is given by
$$\|u\|_{W^{1,2}_D(\Omega, \mu_\Omega^*)}^2 =  \|u\|_{L^2(\Omega)}^2 + \|\nabla u\|_{L^2(\Omega, \mu_\Omega^*)}^2.$$
As before, here we write
$$\|f\|_{L^2(\Omega, \mu_\Omega^*)}^2 = \int_\Omega |f|^2 \mu_\Omega^*\,dx.$$
 Note that $\mu_\Omega^*$ appears as a weight only in the gradient, the $L^2(\Omega)$-norm remains unweighted. 

We record the following properties. For the general theory of \emph{Muckenhoupt weights} we refer to \cite[Chapter 9]{Gra}.

\begin{lemma} \label{lem:weight} \begin{itemize}
 \item[\text{(a)}] The weight $\mu_\Omega^*$ belongs to the Muckenhoupt class $\mathcal A_2$.
 \item[\text{(b)}] One has the continuous embedding $W^{1,2}_D(\Omega, \mu_\Omega^*) \subset W^{1,1}(\Omega)$.
 \item[\text{(c)}] $W^{1,2}_D(\Omega, \mu_\Omega^*)$ is a Hilbert space with scalar product
$$(u,v)_{W^{1,2}_D(\Omega, \mu_\Omega^*)} = (u,v)_{L^2(\Omega)} + \int_{\Omega} (\nabla u,\nabla v)\, \mu_\Omega^*\,dx.$$
 \end{itemize}
\end{lemma}
\begin{proof} Assertion (a) follows from our assumption $0\leq\gamma < k$, see \cite[Lemma 2.3]{FS}. Using H\"older's inequality, it is straightforward to check that $L^2(\Omega, \mu_\Omega^*) \subset L^1(\Omega)$ (see also \cite[Exercise 9.3.6]{Gra}), which yields (b). Then (c) follows from (a). \end{proof}

To prove the continuity of $\tr_{\Gamma_{\dyn}}$ and $\tr_\Sigma$ on $W^{1,2}_D(\Omega, \mu_\Omega^*)$ we start with an extension operator of this space to $W^{1,2}(\R^d, \mu_\Omega^*)$. Here the norm is given by
$$\|u\|_{W^{1,2}(\R^d, \mu_\Omega^*)}^2 = \|u\|_{L^2(\R^d)}^2 + \|\nabla u\|_{L^2(\R^d, \mu_\Omega^*)}^2.$$

\begin{lemma} \label{lem:extension} There is a continuous extension operator $\mathcal E:W_D^{1,2}(\Omega,\mu_\Omega^*)
  \to W^{1,2}(\R^d, \mu_\Omega^*).$ 
For any $u\in W_D^{1,2}(\Omega,\mu_\Omega^*)$ we have that 
$\supp \mathcal E u \subset B(0,2R)$, where $R = \sup\{|x|:x\in \Omega\}$.
\end{lemma}
\begin{proof} \emph{Step 1}. From Assumption \ref{assu:geo} we find a finite open covering $\bigcup_{\alpha=1}^N V_\alpha \subset B(0,2R)$ of $\overline{\Omega}$ with the following properties. For $\alpha = 1,\ldots, N_\Omega$ the sets $V_\alpha$ are strictly contained in $\Omega$; for $\alpha = N_\Omega +1 ,\ldots, N_D$ we have $V_\alpha \cap \Gamma_D \neq \emptyset$ and $V_\alpha \cap (\overline{\Gamma_N\cup \Gamma_\dyn}) = \emptyset$;   for $\alpha = N_D+1,..., N$ there is a bi-Lipschitz map $\varphi_\alpha$ from $V_\alpha$ to the open unit cube $Q$ in $\R^d$  such that 
$$\varphi_\alpha(\Omega\cap V_\alpha) = Q_-, \qquad \varphi_\alpha(\partial \Omega \cap V_\alpha) = Q_0,$$
where $Q_-\subset Q$ is the open lower half-cube in $\R^d$ and $Q_0 = \{x\in Q: x_d = 0\}$. We further take a smooth partition of unity $(\psi_\alpha)_\alpha$ for $\overline{\Omega}$ subordinate to the cover $\bigcup_{\alpha} V_\alpha$, i.e., such that $\supp \psi_\alpha$ is contained in $V_\alpha$.

\emph{Step 2}. For any  $u\in C_D^\infty(\Omega)$ and $\alpha = N_D+ 1,\ldots, N$ we have that $\psi_\alpha u$ is compactly supported in $\overline{\Omega} \cap V_\alpha$. Choose an open subcube $\widetilde Q\subset Q$ such that $\varphi_\alpha(\supp \psi_\alpha) \subset \widetilde Q$.  Then  $W_\alpha = \varphi_\alpha^{-1}(\widetilde{Q}_-)$ is a domain with Lipschitz boundary which contains $\supp \psi_\alpha$. Finally, take smooth cut-off functions $\phi_\alpha$ such that $\phi_\alpha \equiv 1$ on $\supp \psi_\alpha$ and $\supp \phi_\alpha \subset V_\alpha$. 

\emph{Step 3}. Now for $u\in C_D^\infty(\Omega)$ we define  $\mathcal E u$ by
$$\mathcal E u = \sum_{\alpha = 1}^{N_\Omega} \psi_\alpha u + \sum_{\alpha = N_\Omega +1}^{N_D} \mathcal E_\alpha( \psi_\alpha u) + \sum_{\alpha = N_D + 1}^{N} \phi_\alpha \mathcal E_\alpha (\psi_\alpha u|_{W_\alpha}),$$
where the extensions $\mathcal E_\alpha$ are given as follows. For $\alpha = N_\Omega+1,\ldots,N_D$ we define $\mathcal E_\alpha( \psi_\alpha u)$ as the trivial extension by zero of $\psi_\alpha u$ from $V_\alpha\cap \Omega$ to $\R^d$. Since $V_\alpha \cap (\overline{\Gamma_N\cup \Gamma_\dyn}) = \emptyset$ and $u$ is supported away from $\Gamma_D$, for those $\alpha$ we have
$$\|\mathcal E_\alpha( \psi_\alpha u) \|_{W^{1,2}(\R^d, \mu_\Omega^*)} = \|\psi_\alpha u \|_{W^{1,2}(\Omega, \mu_\Omega^*)} \leq C \|u\|_{W^{1,2}(\Omega, \mu_\Omega^*)}.$$
For $\alpha = N_D+1,\ldots N$ we let $\mathcal E_\alpha : W^{1,2}(W_\alpha,\mu_\Omega^*) \to W^{1,2}(\mathbb{R}^d,\mu_\Omega^*)$ be the extension operator from \cite{Chua} for the Lipschitz domain $W_\alpha$. Then 
\begin{align*}
 \|\phi_\alpha \mathcal E_\alpha( \psi_\alpha u|_{W_\alpha}) \|_{W^{1,2}(\R^d, \mu_\Omega^*)} &\, \leq C \|\mathcal E_\alpha( \psi_\alpha u|_{W_\alpha}) \|_{W^{1,2}(\R^d, \mu_\Omega^*)} \leq  C \|\psi_\alpha u|_{W_\alpha} \|_{W^{1,2}(W_\alpha, \mu_\Omega^*)} \\
 &\, = C \|\psi_\alpha u\|_{W^{1,2}(\Omega, \mu_\Omega^*)} \leq  C \|u\|_{W^{1,2}(\Omega, \mu_\Omega^*)}.
\end{align*}
Therefore $\mathcal E$ extends continuously from $C_D^\infty(\Omega)$ to  $\mathcal E:W_D^{1,2}(\Omega,\mu_\Omega^*) \to W^{1,2}(\R^d, \mu_\Omega^*),$ which gives the desired extension operator. \end{proof}

In a next step we prove Sobolev embeddings of $W^{1,2}(\R^d, \mu_\Omega^*)$ into unweighted Slobodetskii spaces $W^{\theta,q}(\R^d)$, by using the criteria derived in \cite{HS}.

\begin{proposition} \label{prop:Sob-weight} Assume  $q \in [2,\infty)$ and $\theta\in (0,1)$  are such that $1 - \frac{d+\gamma}{2} \geq \theta - \frac{d}{q}.$ Then 
$$W^{1,2}(\R^d, \mu_\Omega^*) \subset W^{\theta,q}(\R^d).$$
\end{proposition}

\begin{proof} \emph{Step 1.} Let $B_{2,2}^1(\R^d,\mu_\Omega^*)$ be the Besov space with respect to the weight $\mu_\Omega^*$. Since $\mu_\Omega^*$ belongs to the Muckenhoupt class $\mathcal A_2$ by Lemma \ref{lem:weight}, it follows from Remark 1.7 and Proposition 1.8 of \cite{HS} that 
$$W^{1,2}(\R^d,\mu_\Omega^*) \subset  B_{2,2}^1(\R^d,\mu_\Omega^*).$$
Moreover, $W^{\theta,q}(\R^d) = B_{q,q}^\theta(\R^d)$ for $\theta\in (0,1)$ by \cite[Section 2.3.1]{Triebel}. The asserted embedding will thus be a consequence of
\begin{equation}\label{Besov}
	B_{2,2}^1(\R^d,\mu_\Omega^*) \subset B_{q,q}^\theta(\R^d).
	\end{equation}

\emph{Step 2.} We derive this embedding from the sufficient condition given in \cite[Proposition 2.1(i)]{HS}. Let $Q(x,r)$ be the cube in $\R^d$ with edges parallel
to the coordinate axes, centered at $x \in \R^d$ with edge length $r>0$. According to \cite{HS}, \eqref{Besov} holds true if we show that
$$\sup_{l\in \N_0, m\in \Z^d} 2^{-l(1-\theta + \frac{d}{q})}\Big ( \int_{Q(2^{-l}m,2^{-l})} \text{dist}(x,S)^\gamma\,dx\Big)^{-1/2} <\infty.$$
By the assumption $1 - \frac{d+\gamma}{2} \geq \theta - \frac{d}{q}$, this will be a consequence of the estimate
\begin{equation}\label{Q_estimate}
\int_{Q(2^{-l}m,2^{-l})} \text{dist}(x,S)^\gamma\,dx\geq c 2^{-l(d+\gamma)},\qquad l\in \N, \qquad m\in \Z^d,
  \end{equation}
where $c>0$ is independent of $l$ and $m$. In the sequel we prove \eqref{Q_estimate}.

\emph{Step 3.}  Since $S$ is Lipschitzian, there is a tube $S_\kappa$ of width $\kappa >0$ around $S$ such that every $Q(2^{-l}m,2^{-l})\subset S_\kappa$ 
lies in a neighbourhood $V$ of $S$ which is mapped to the unit cube in $\mathbb{R}^d$ by a bi-Lipschitz map $\psi$ such that $ \psi(S\cap V) = (-1,1)^{d-k} \times \{0_{\R^k}\}$. 

Choose $l_0\in \N$ such that $2^{-l_0 \gamma}+2^{-l_0} \leq \kappa$. We claim that it suffices to prove \eqref{Q_estimate} for $l\geq l_0$ and $m$ such that $Q(2^{-l}m,2^{-l})\subset S_\kappa$, where $c$ is independent of those $l$ and $m$.

Assume \eqref{Q_estimate} is proved for those $l$ and $m$. Let $l\geq l_0$ and $m$ be such that $Q(2^{-l}m,2^{-l})$ is not contained in $S_\kappa$. Then we trivially have
$$\int_{Q(2^{-l}m,2^{-l})} \text{dist}(x,S)^\gamma\,dx \geq c 2^{-ld}2^{-l_0\gamma}\geq c 2^{-l(d+\gamma)}.$$ 
This yields \eqref{Q_estimate} for $l\geq l_0$ and arbitrary $m$. Let $l<l_0$. Then
$$\int_{Q(2^{-l}m,2^{-l})} \text{dist}(x,S)^\gamma\,dx \geq \int_{Q(2^{-l_0}(2^{l_0-l}m),2^{-l_0})} \text{dist}(x,S)^\gamma\,dx \geq c 2^{-l_0(d+\gamma)} \geq \tilde c 2^{-l(d+\gamma)},$$
where $\tilde c = 2^{-l_0(d+\gamma)}$ is independent of $l$ and $m$.

\emph{Step 4.} It remains to prove \eqref{Q_estimate} for $l\geq l_0$ and $m$ such that $Q(2^{-l}m,2^{-l})\subset S_\kappa$. The integral in \eqref{Q_estimate} transforms as 
$$ \int_{Q(2^{-l}m,2^{-l})} \text{dist}(x,S)^\gamma\,dx = \int_{\psi(Q(2^{-l}m,2^{-l}))} \text{dist}(\psi^{-1}(y),S)^\gamma |\det \psi'(y)|^{-1}\, dy,$$
where $|\det \psi'|^{-1} \geq c$ can be uniformly chosen by compactness of $S$. From the bi-Lipschitz property of $\psi$ 
    it follows that $\text{dist}(\psi^{-1}(y),S) \simeq \text{dist}(y,\psi(S))$. Since $\psi(S) \subset \R^{d-k}\times \{0_{\R^k}\}$, we thus get
    $$ \int_{Q(2^{-l}m,2^{-l})} \text{dist}(x,S)^\gamma\,dx \geq
    c\int_{\psi(Q(2^{-l}m,2^{-l}))} \big( \vert y_{d-k+1} \vert^\gamma + \ldots + \vert y_d \vert^\gamma\big) \, dy.$$
Again the bi-Lipschitz property of $\psi$ yields $\delta > 0$, independent of $l$ and $m$, such that $Q(\psi(2^{-l}m), \delta 2^{-l})$ is contained in $\psi(Q(2^{-l}m,2^{-l}))$. It therefore remains to estimate 
$$\int_{Q(\psi(2^{-l}m), \delta 2^{-l})} \big( \vert y_{d-k+1} \vert^\gamma + \ldots + \vert y_d \vert^\gamma\big) \, dy = \delta^{d-1}2^{-l(d-1)} \sum_{j=0}^{k-1}  \int_{\psi_j(2^{-l}m)-\delta2^{-l}}
     ^{\psi_j (2^{-l}m)+\delta 2^{-l}} \vert \tau \vert^\gamma \, d\tau.$$
For each $j$, here the integral is given by 
$$ \eta(s,t) := \frac{1}{\gamma + 1} (\text{sign}(s+t)\vert s+t \vert^{\gamma + 1} - \text{sign}(s-t) \vert s-t \vert^{\gamma +1}),$$
     where $s=\psi_j(2^{-l}m)\in \mathbb{R}$  and $t=\delta 2^{-l} > 0$. By distinguishing the three cases $s\geq t$, 
    $s \in (-t,t)$  and $s \leq -t$ and using the triangle inequality for the $(\gamma+1)$-norm in $\R^2$, we see that 
    $\eta(s,t)\geq c t^{\gamma+1}$, where $c$ is independent of $s$. We thus obtain the estimate
    $$\int_{Q(\psi(2^{-l}m), \delta 2^{-l})} \big( \vert y_{d-k+1} \vert^\gamma + \ldots + \vert y_d \vert^\gamma\big) \, dy \geq c2^{-l(d+\gamma)},$$
  independently of $m$, and this gives \eqref{Q_estimate}.\end{proof}

\begin{rem}A scaling argument gives necessary conditions on the parameters for embedding of the type given in Proposition \ref{prop:Sob-weight} to hold, at least in the model case 
$S = \R^{d-k}\times \{0_k\}$, where
 \begin{equation} \text{dist}(x,S)^\gamma \sim |x_1|^\gamma + ... + |x_k|^\gamma.\end{equation} \label{equNorm}
 Assuming $\|u\|_{W^{\theta,q}(\R^d)} \leq C \|u\|_{W^{1,2}(\R^d; \text{dist}(\cdot,S)^\gamma)}$ for a constant $C$ independent of $u$, replacing $u$ by $u(\lambda\cdot)$ with $\lambda > 0$ and rescaling $y = \lambda x$ such that $dx = \lambda^{-d}dx$, we obtain
$$\lambda^{-\frac{d}{q}}\|u\|_{L^q(\R^d)} + \lambda^{\theta-\frac{d}{q}}[u]_{W^{\theta,q}(\R^d)}\leq C\big ( \lambda^{-\frac{d}{2}}\|u\|_{L^2(\R^d)} + \lambda^{1-\frac{d+\gamma}{2}}\|\nabla u\|_{L^2(\R^d;\text{dist}(\cdot,S)^\gamma dx)}\big).$$
Letting $\lambda\to \infty$, this shows that for any $\theta\in (0,1)$ and $q \geq 2$ the condition $1-\frac{d+\gamma}{2}\geq \theta - \frac{d}{q}$ is necessary. \end{rem}

We combine the above results to obtain the following properties of the traces.

\begin{proposition}\label{prop:trace-weight} For $1<r < \frac{2(d-1)}{d+\gamma -2}$ the trace operators $\emph{\tr}_{\Gamma_{\emph{\dyn}}}$ and $\emph{\tr}_\Sigma$ are continuous and compact maps
$$\emph{\tr}_{\Gamma_{\emph{\dyn}}}: W^{1,2}(\Omega, \mu_\Omega^*) \to L^r(\Gamma_{\emph{\dyn}}, d\mathcal H_{d-1}), \qquad \emph{\tr}_{\Sigma}: W^{1,2}(\Omega, \mu_\Omega^*) \to L^r(\Sigma, d\mathcal H_{d-1}).$$
\end{proposition}
\begin{proof} We consider $\Sigma$, the arguments for $\Gamma_{\dyn}$ are the same. Let $\mathcal E$ be the extension operator for $W^{1,2}(\Omega, \mu_\Omega^*)$ from Lemma \ref{lem:extension}.  As in the proof of \cite[Proposition 2.8]{EMR} one can show that $\tr_{\Sigma} = \tr_{\Sigma} \mathcal E$.  Proposition \ref{prop:Sob-weight} together with the support property of $\mathcal E$ implies that there is $\varepsilon > 0$  such that $\mathcal E$ maps $W^{1,2}(\Omega, \mu_\Omega^*)$ compactly into $W^{1/r+\varepsilon,r}(\R^d)$ for $r>1$, provided $1-\frac{d+\gamma}{2} > \frac{1-d}{r}$. Since $d \geq 2$ and $\gamma >0$  we have $1-\frac{d+\gamma}{2} < 0$, such that this inequality is equivalent to $r < \frac{2(d-1)}{d+\gamma -2}$. Now \cite[Lemma 2.7]{EMR} implies that $\tr_\Sigma$ maps $W^{1/r+\varepsilon,r}(\R^d)$ continuously into $L^r(\Sigma, d\mathcal H_{d-1})$ for those $r$. Altogether, $\tr_\Sigma$ is continuous and compact.
\end{proof}


\subsection{The operators $A_p$ on $\mathbb L^p$} We modify $\dom(\ft)$ from Definition \ref{def:domt} to take into account the degeneracy of the diffusion coefficient $\mu_\Omega$. We set
$$(u,v)_{\dom(\mathfrak t)} = (u,v)_{W^{1,2}(\Omega,\mu_\Omega^*)} + \int_{\Gamma_\dyn} \big(\nabla_{\Gamma_\dyn} u, \overline{\nabla_{\Gamma_\dyn} v}\big) \,\mu_{\Gamma_\dyn}^*\, d\mathcal H_{d-1} + \int_{\Sigma} \big(\nabla_\Sigma u, \overline{\nabla_\Sigma v}\big) \,\mu_{\Sigma}^* \,d\mathcal H_{d-1},$$
and define as before $\dom(\ft)$ as the completion of $C_D^\infty(\Omega)$ with respect to the corresponding Hilbert norm $\|\cdot\|_{\dom(\ft)}$. It is now given by
$$\|u\|_{\dom(\ft)}^2 = \|u\|_{W^{1,2}(\Omega,\mu_\Omega^*)}^2 + \|\nabla_{\Gamma_{\dyn}} u\|_{L^2(\Gamma_{\dyn}, \mu_{\Gamma_{\dyn}}^*)}^2 + \|\nabla_{\Sigma} u\|_{L^2(\Sigma, \mu_{\Sigma}^*)}^2.$$

Recall that the map $\J$ is for $u\in C_D^\infty(\Omega)$ given by $\J u = (u,u_{\Gamma_{\dyn}}, u_\Sigma)$. In the following we distinguish between the cases when the surface $S$, where the bulk diffusion degenerates, is away from $\Gamma_{\dyn}$ and $\Sigma$, and where the relation between these sets is arbitrary. In the second case we have to restrict to $\gamma <1$ for the distance exponent to obtain the continuity of $\J$ into $\L^2$.

\begin{lemma} Assume either $0<\gamma < d-k$ and $S\cap (\overline{\Gamma_{\emph{\dyn}}\cup \Sigma}) = \emptyset$ (Case $(A)$), 
  or assume $0< \gamma < 1$ (Case $(B)$). Then $\mathfrak J: \dom(\mathfrak t)\to \mathbb L^2$ is continuous and  has dense range. If (additionally) $0 < \gamma < 2$, then $\J$ is compact.
\end{lemma}

\begin{proof} \emph{Step 1.} Since $\dom(\ft) \subset W^{1,2}_D(\Omega,\mu_\Omega^*)$, for continuity and compactness it suffices to consider $\J$ on $W^{1,2}_D(\Omega,\mu_\Omega^*)$ instead of $\dom(\mathfrak t)$.

By definition we have $W^{1,2}_D(\Omega,\mu_\Omega^*)\subset L^2(\Omega)$. We claim that the latter embedding is also compact if $\gamma < 2$. Decompose the embedding into the extension $\mathcal E$ to $W^{1,2}(\R^d,\mu_\Omega^*)$ from Lemma \ref{lem:extension} and the restriction to $\Omega$. By Proposition \ref{prop:Sob-weight} we have $W^{1,2}(\R^d,\mu_\Omega^*) \subset W^{\theta,2}(\R^d)$ for some $\theta > 0$, provided $\gamma <2$. The support property yields that $\mathcal E$ is compact if $\theta$ is chosen slightly smaller. Hence $W^{1,2}_D(\Omega,\mu_\Omega^*)$ embeds compactly into  $L^2(\Omega)$ for $\gamma <2$.

\emph{Step 2.} We show that the traces at $\Gamma_{\dyn}$ and $\Sigma$ are continuous and compact from $W^{1,2}_D(\Omega,\mu_\Omega^*)$ into $L^2(\Gamma_{\dyn})$ and $L^2(\Sigma)$, respectively. Assume $\gamma < 1$. Then $\frac{2(d-1)}{d+\gamma -2} > 2$, and the assertion follows from Proposition \ref{prop:trace-weight}. Next assume $S\cap (\overline{\Gamma_{\dyn}\cup \Sigma}) =  \emptyset$. Choose a smooth cut-off $\psi$ such that $\psi \equiv 0$ on $S$ and $\psi\equiv 1$ in a neighbourhood of $\overline{\Gamma_{\dyn}\cup \Sigma}$. Then $\tr_{\Sigma} u = \tr_{\Sigma} (\psi u)$ for all $u\in W^{1,2}_D(\Omega,\mu_\Omega^*)$. The multiplication with $\psi$ is continuous from $W^{1,2}_D(\Omega,\mu_\Omega^*)$ into the unweighted space $W^{1,2}_D(\Omega)$, and $\tr_{\Sigma}$ is continuous and compact from $W^{1,2}_D(\Omega)$ to $L^2(\Sigma)$ by \cite[Lemma 2.10]{EMR}, analogously for $\tr_{\Gamma_{\dyn}}$.

\emph{Step 3.} By the proof of \cite[Lemma 2.10]{EMR} we have that $\J C_D^\infty(\Omega)$ is dense in $\L^2$. Hence $\J W^{1,2}_D(\Omega,\mu_\Omega^*)$ is dense since $C_D^\infty(\Omega) \subset W^{1,2}_D(\Omega,\mu_\Omega^*)$.\end{proof}

Now one can argue in the same way as in Lemma \ref{lem:t-J} to show that the sesquilinear form 
$$\mathfrak t(u,v) = \int_\Omega (\mu_\Omega \nabla u, \overline{\nabla v})\,dx + \int_{\Gamma_\dyn} \big (\mu_{\Gamma_\dyn}\nabla_{\Gamma_\dyn} u, \overline{\nabla_{\Gamma_\dyn} v}\big) \,d\mathcal H_{d-1} + \int_{\Sigma} \big(\mu_\Sigma\nabla_\Sigma u, \overline{\nabla_\Sigma v}\big) \,d\mathcal H_{d-1}$$
extends continuously from $C_D^\infty(\Omega)$ to $\dom(\ft)$, and that it is $\J$-elliptic.
Therefore, as in Proposition \ref{prop:AE1} we obtain a closed and densely defined operator $A_2$ associated to $\ft$, 
  which is the negative generator of an analytic $C_0$-semigroup $T_2(\cdot)$ on $\L^2$. 
  In order to show that $T_2(\cdot)$ is $\L^\infty$-contractive, it suffices to see that as in the proof of Proposition \ref{p:markovian}, $T_2(\cdot)$ 
  leaves $\L^2_{\mathbb{R}}$ and $\mathcal C$ invariant. 


Then, as in Section \ref{sec:extension}, the semigroup $T_2(\cdot)$ on $\L^2$ extends consistently to $T_p(\cdot)$ on $\L^p$ for $p\in [1,\infty]$, and for the generators $A_p$ and the relaxation coefficient $\zeta$ we obtain our main result.

\begin{theorem}\label{thm:mr2} Assume either $0<\gamma < d-k$ and $S\cap (\overline{\Gamma_{\emph{\dyn}}\cup \Sigma}) = \emptyset$ (Case $(A)$), 
  or assume $0< \gamma < 1$ (Case $(B)$). Then for each $p\in (1,\infty)$ the operator $\zeta^{-1} A_p$ with domain $\dom(A_p)$ admits a bounded holomorphic functional calculus on $\mathbb L^p$, with angle strictly smaller than $\frac{\pi}{2}$.  As a consequence, $\zeta^{-1} A_p$ enjoys maximal parabolic $L^s$-regularity for all $s\in (1,\infty)$ and $-\zeta^{-1} A_p$ generates an analytic $C_0$-semigroup on $\L^p$. Furthermore, the fractional power domains are given by complex interpolation, i.e.,
$$\dom(A_p^\theta) = [\mathbb L^p, \dom (A_p)]_\theta,\qquad \theta\in [0,1].$$ 
The resolvent of $\zeta^{-1} A_p$ is compact if $\gamma < 2$. \end{theorem}


\section{Embeddings for fractional power domains}\label{sec:emb}

Let $A_p$ be the operator from Theorem \ref{thm:1} or \ref{thm:mr2}. In this section we investigate conditions on $p\in (2,\infty)$ and $\theta\in (0,1)$ such that for the domain of the fractional power $A_p^\theta$ we have
\begin{equation}\label{eq:emb}
\dom(A_p^\theta) \hookrightarrow \mathbb L^\infty.
\end{equation}
We in particular aim to quantify the conditions in dependence on whether diffusion is degenerate or not, and where it degenerates.

Our motivation are semilinear versions of \eqref{e-parabol}--\eqref{e-initial}, 
  i.e., where the right-hand side $(f_{\Omega}, f_{\Gamma_{\dyn}}, f_\Sigma)$ depends nonlinearly on the solution itself. 
  If \eqref{eq:emb} holds true, then the  Nemytzkii operator induced by a  nonlinearity is well-defined on $\dom(A_p^\theta)$ with values in $\L^p$, 
  which in principle allows to apply the standard theory for semilinear parabolic equations to obtain local-in-time well-posedness 
  (see the introduction for further references).

The key to the embedding \eqref{eq:emb} is the regularity of the image of $\J$.

\begin{lemma}\label{l:sg-est} Let $p,r\in (2,\infty)$ and $\theta \in (0,1)$ such that $\theta > \frac{r}{(r-2)p}$. Assume 
\begin{equation}\label{eq:Lr}
\mathfrak{J}\dom(\mathfrak t) \subset \mathbb{L}^r.
\end{equation}
Then $\dom(A_p^\theta) \subset \mathbb L^\infty$.
\end{lemma}

\begin{proof} Let $T_p(\cdot)$ be the semigroup on $\L^p$ generated by $-A_p$. The arguments given in the proof of \cite[Lemma 2.19]{EMR} show that there is $C > 0$ such that 
$$\|e^{-t} T_2(t) \varphi\|_{\mathbb L^\infty} \leq C t^{-\frac{r}{(r-2)2}} \|\varphi \|_{\mathbb L^2},\qquad t > 0,\qquad \varphi\in \mathbb L^2.$$
Interpolating this inequality with the $\L^\infty$-contractivity of $T_2(\cdot)$, we obtain that 
\begin{equation}\label{eq:esti-123}
 \|e^{-t} T_p(t) \varphi\|_{\mathbb L^\infty} \leq C t^{-\frac{r}{(r-2)p}} \|\varphi \|_{\mathbb L^p},\qquad t > 0,\qquad \varphi\in \mathbb L^p.
\end{equation}
Since $1+A_p$ is invertible, we have that $$u\mapsto \|(A_p+1)^\theta u\|_{\L^p}$$ defines an equivalent norm on $\dom(A_p^\theta)$. For $\theta\in (0,1)$ it is further well-known that 
$$
(A_p+1)^{-\theta} = C_\theta \int_0^\infty t^{\theta-1} e^{-t} T_p(t) \,dt.
$$
Using \eqref{eq:esti-123} for $t\in (0,1)$ and the contractivity of $T_p(\cdot)$ for $t> 1$, for $u \in \dom(A_p^\theta)$ we obtain
$$\|u\|_{\L^\infty}  \leq C \|u\|_{\dom(A_p^\theta)} \int_0^1 t^{\theta-1 - \frac{r}{(r-2)p}}  \,  dt +  C \|u\|_{\dom(A_p^\theta)} \int_1^\infty  e^{-t} \,  dt.$$
Here the first integral is finite if $\theta > \frac{r}{(r-2)p}$. In this case the embedding $\dom(A_p^\theta) \subset \mathbb L^\infty$ follows.
\end{proof}

In the sequel we determine $r_0 > 2$ as large as possible such that \eqref{eq:Lr} holds for all $2<r<r_0$. Since 
$$\L^r = L^r(\Omega)\oplus L^r(\Gamma_{\dyn}) \oplus L^r(\Sigma),$$
the number $r_0$ depends on how large $r$ can be such that
$$\dom(\mathfrak t) \subset L^r(\Omega), \qquad \tr_{\Gamma_{\dyn}}:\dom(\mathfrak t) \to L^r(\Gamma_{\dyn}),
  \qquad \tr_{\Sigma}:\dom(\mathfrak t) \to L^r(\Sigma),$$
are simultaneously continuous. In turn, this depends on whether the bulk diffusion degenerates or not,
    if it degenerates at $\overline{\Gamma_{\dyn}\cup \Sigma}$ where traces are taken, and {\em where} the surface diffusion on $\Gamma_{\dyn}$
    and $\Sigma$ degenerates. 

It follows from Lemma \ref{lem:extension} and Proposition \ref{prop:Sob-weight} that 
$$\dom(\mathfrak t) \subset W_D^{1,2}(\Omega) \subset L^{r}(\Omega)$$
for $r < r_\Omega := \frac{2d}{(d+\gamma -2)_+}$. If $S = \emptyset$ or $S \cap \overline{\Gamma_\dyn\cup \Sigma} = \emptyset$ 
(Case $(\mathrm{A})$),
    then by \cite[Proposition 2.8]{EMR} the traces are continuous from  $\dom(\mathfrak t) \subset W_D^{1,2}(\Omega)$ 
    into $L^r(\Gamma_{\dyn})$ and $L^r(\Sigma)$ for all $r < r_{\text{tr}} := \frac{2(d-1)}{(d-2)_+}$. 
    In case $S \cap \overline{\Gamma_\dyn\cup \Sigma} \neq \emptyset$ (Case $(\mathrm{B})$), where in Theorem \ref{thm:mr2} it is assumed that $\gamma < 1$, Proposition \ref{prop:trace-weight} shows that the traces are continuous only for $r < r_{\text{tr},\gamma}:= \frac{2(d-1)}{(d+\gamma-2)_+}$.

The regularity of the traces improves if surface diffusion is present. 
Assume that the surface diffusion is uniformly nondegenerate, i.e., $\mu_{\Gamma_\dyn}^*, \mu_\Sigma^* \geq \eta > 0$. Then the traces belong to $W^{1,2}(\Gamma_{\dyn})$ and $W^{1,2}(\Sigma)$. By Sobolev embeddings, the traces are thus continuous into $L^r$ for $r < r^*_{\text{tr}} := \frac{2(d-1)}{(d-3)_+}$. Observe that $r^*_{\text{tr}}> r_{\text{tr}}$, which quantifies the regularity improvement obtained from surface diffusion. Finally, assume that $S \cap \overline{\Gamma_\dyn\cup \Sigma} \neq \emptyset$ and that $\mu_{\Gamma_\dyn}^*, \mu_\Sigma^* \geq \eta > 0$ in a neighbourhood of $S \cap \overline{\Gamma_\dyn\cup \Sigma}$. Then the traces belong to $W^{1,2}$ in this neighbourhood, such that they belong to $L^r$ for $r < \min(r_{\text{tr}}, r_{\text{tr}}^*) = r_{\text{tr}}$. This improves the case without surface diffusion on the critical set  $S \cap \overline{\Gamma_\dyn\cup \Sigma}$ since $r_{\text{tr}} > r_{\text{tr},\gamma}.$

Now the number $r_0$ can be chosen as the minimum of $r_\Omega$ and $r_{\text{tr}}$, $r_{\text{tr},\gamma}$ or $r_{\text{tr}}^*$ according to the cases described above. The following figure gives an overview.

\includegraphics[scale=0.57]{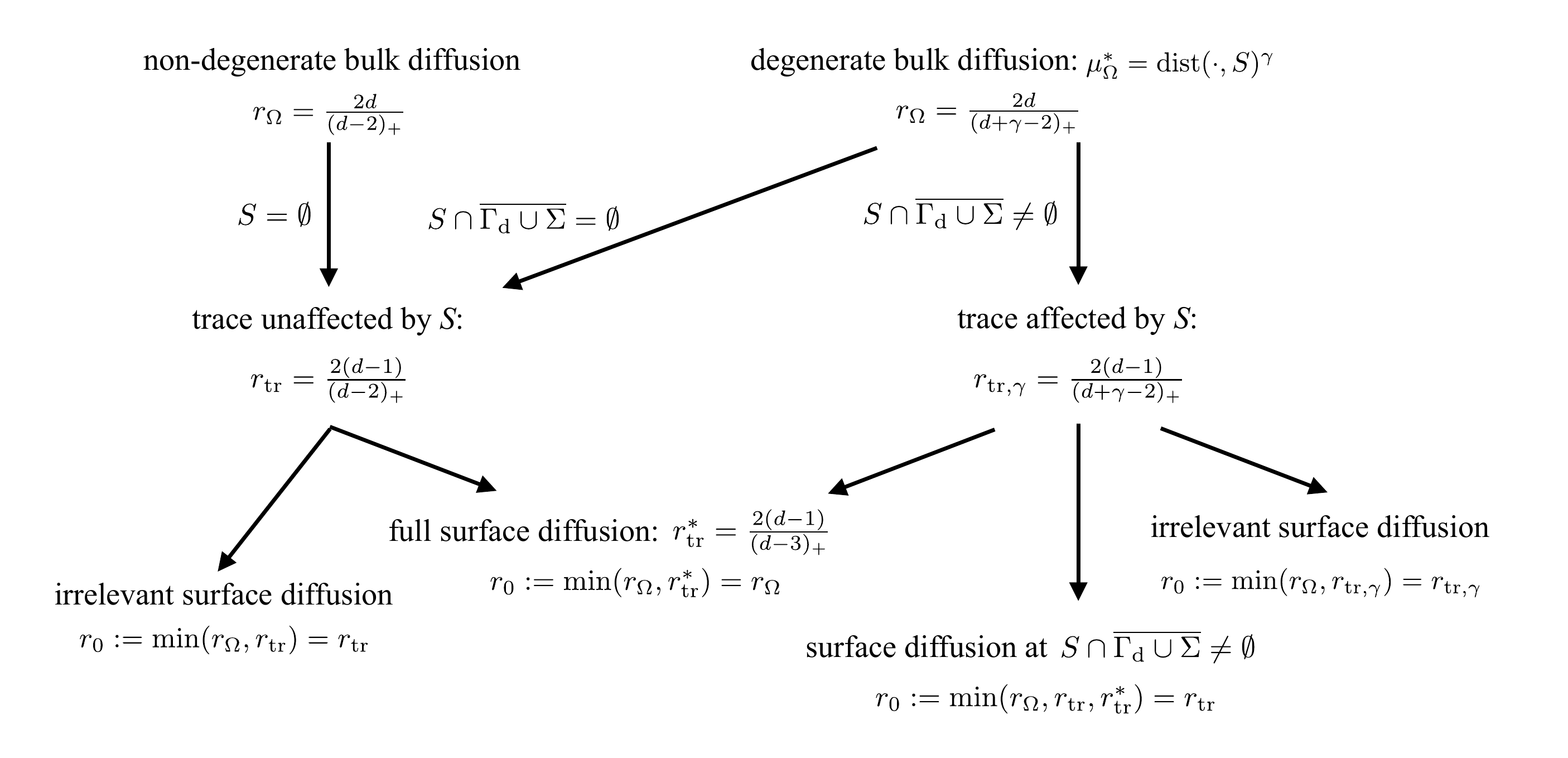}

One can check that if $0\leq \gamma <1$, in any case we have $r_0>2$. Together with 
Lemma \ref{l:sg-est} we thus have the following result.

\begin{theorem} Assume $0\leq \gamma <1$. Then there are $\theta_0\in (0,1)$ and $p_0\in 
(2,\infty)$ such that $\dom(A_p^\theta) \hookrightarrow \mathbb L^\infty$ for all 
$\theta \in (\theta_0,1)$ and $p\in (p_0,\infty)$.
\end{theorem}

It is interesting to note that if diffusion is nowhere degenerate, then one can take 
$r_0=\frac{2d}{(d-2)_+}$. In this case, by Lemma \ref{l:sg-est} we have $\dom(A_p^\theta)
 \hookrightarrow \mathbb L^\infty$ provided 
$$2\theta > \frac{d}{p}.$$
This is precisely the optimal relation for the embedding of $H^{2\theta,p}$ into $L^\infty$.
 In a smooth situation one indeed expects that $\dom(A_p) \subset H^{2,p}(\Omega)$ and thus 
$\dom(A^\theta_p) \subset H^{2\theta,p}(\Omega)$, which shows that the above considerations
 are optimal at least in this case.

\subsection*{Acknowledgements}
	The authors thank Dorothee Knees and Maria Specovius--Neugebauer for many helpful comments and suggestions.


\begin{thebibliography}{99}


\bibitem{AKS} H. Abels, M. Krbec, K. Schumacher.
On the trace space of a {S}obolev space with a radial weight,
   \emph{J. Funct. Spaces Appl.} 6(3) (2008), 259--276.


\bibitem{AmEs} H. Amann, J. Escher. 
\emph{Analysis II $\&$ III}, 
Birkh\"auser, Basel (2006).

\bibitem{AIMT} F. Andreu, N. Igbida, J. M. Maz\'on, J. Toledo. Renormalized solutions for degenerate
 elliptic-parabolic problems with nonlinear dynamical boundary conditions and $L^1$-data. 
\emph{J. Differential Equations} 244(11) (2008), 2764-2803.


\bibitem{AE}
W. Arendt, A. F. M. ter Elst. Sectorial forms and degenerate
  differential operators, \emph{J. Operator Theory} 67 (2012), 33--72.
  
  \bibitem{AQRB} J. M. Arrieta, P. Quittner, A. Rodriguez-Bernal. Parabolic problems with 
nonlinear dynamical boundary conditions and singular initial data, \emph{Differential Int. Eq.} 14 (2001),
 1487--1510.
  
\bibitem{BC} J. von Below, C. Coster. A qualitative theory for parabolic problems under 
dynamical boundary conditions. \emph{J. Inequal. Appl.} 5 (2000), 467--486.

\bibitem{cia}
P.G.~Ciarlet, \emph{The finite element method for elliptic problems},
Studies in Mathematics and its Applications, North Holland, Amsterdam/
New York/ Oxford, 1979

\bibitem{Chua} S.-K. Chua. 
Extension theorems on weighted Sobolev spaces,
\emph{Indiana Univ. Math. J.} 41(4) (1992), 1027--1076.

\bibitem{Dore} G. Dore. Maximal regularity in $L^p$ spaces for an abstract Cauchy problem. 
\emph{Adv. Differential Equations} 5 (2000), 293--322.

\bibitem{Duong} X. T. Duong. $H^\infty$-functional calculus of second order elliptic partial 
differential operators on $L^p$- spaces, \emph{Proc. CMA Canberra} 24 (1989), 91--102.

\bibitem{EMR} A. F. M. ter Elst, M. Meyries, J. Rehberg. 
Parabolic equations with dynamical boundary conditions and source terms on interfaces, to appear
 in \emph{Ann. Mat. Pura Appl.}

\bibitem{ER} A. F. M. ter Elst, J. Rehberg. 
$L^\infty$-estimates for divergence operators on bad domains,
\emph{J. Math. Anal. Appl.} 10(2) (2012),  207--214.

\bibitem{Esc} J. Escher. Quasilinear parabolic systems with dynamical boundary
  conditions, {\em Comm. Partial Differential Equations} 18 (1993), 1309--1364.

\bibitem{EG} L. C. Evans, R. F. Gariepy. 
\emph{Measure theory and fine properties of functions}, Studies in Advanced Mathematics,
CRC Press (1991).

\bibitem{FS} R. Farwig, H. Sohr. 
Weighted $L^q$-theory for the Stokes resolvent in exterior domains,
\emph{J. Math. Soc. Japan} 49(2) (1997), 251--288.

\bibitem{FGGR} A. Favini, G. R. Goldstein, J. A. Goldstein, S. Romanelli. The heat equation 
with generalized Wentzell boundary condition. \emph{J. Evol. Eq.} 2 (2002), 1--19.

\bibitem{ggz}
H.~Gajewski, K.~Gr\"oger, K.~Zacharias, \emph{Nichtlineare Operatorgleichungen
und Operatordifferentialgleichungen}, Akademie-Verlag, 1974

\bibitem{Gal1} C. G. Gal. On a class of degenerate parabolic equations with dynamic boundary 
conditions, \emph{J. Differential Equations} 253 (2012), 126--166.

 \bibitem{Gal2} C. G. Gal. Sharp estimates for the global attractor of scalar reaction-diffusion
 equations with a Wentzell boundary condition, \emph{J. Nonlinear Sci.} 22 (2012), 85--106.

\bibitem{Gra} L. Grafakos. 
\emph{Modern Fourier analysis}, Graduate Texts in Mathematics, Springer (2008).

\bibitem{Grisvard} P. Grisvard. 
\emph{Elliptic problems in nonsmooth domains},
SIAM Classics in Applied Mathematics (2011).

\bibitem{HS} D. D. Haroske, L. Skrzypczak. 
Entropy and approximation numbers of embeddings of function spaces with Muckenhoupt weights, I,
\emph{Rev. Mat. Complut.} 21(1) (2008), 135--177.

\bibitem{Hebey} E. Hebey.
\emph{Nonlinear analysis on manifolds: Sobolev spaces and inequalities}, AMS (2000).

\bibitem{henry}
D.~Henry. \emph{Geometric theory of semilinear parabolic equations.} Vol~840 of
Lecture Notes in Mathematics, Springer-Verlag, Berlin-New York, 1981.

\bibitem{Hin} T. Hintermann. Evolution equations with dynamic boundary conditions, 
\emph{Proc. Roy. Soc. Edinburgh Sect. A} 113 (1989), 43--60.

\bibitem{HKR} 
D. H\"omberg, K. Krumbiegel, J. Rehberg. Optimal control of a parabolic equation with dynamic 
boundary condition, \emph{Appl.\ Math.\ Optim.} 67 (2013), 3--31.

\bibitem{IK} N. Igbida, M. Kirane. A degenerate diffusion problem with dynamical boundary conditions.
 \emph{Math. Ann.}, 323(2) (2002), 377--396.

\bibitem{JW} A. Jonsson, H. Wallin. 
\emph{Function spaces on subsets of $\R^n$},
Harwood Academic Publishers (1984).

\bibitem{Jost} J. Jost. 
\emph{Riemannian geometry and geometric analysis (6th ed.)},
Springer (2011).

\bibitem{KW} N. J. Kalton, L. Weis, The $H^\infty$-calculus and sums of closed operators, \emph{Math. Ann.} 321 (2001),
319--345.


\bibitem{Kato} T. Kato. 
\emph{Perturbation theory for linear operators},
Grundlehren der mathematischen Wissenschaften, Springer (1966).


\bibitem{KS} D. Kinderlehrer, G. Stampacchia. 
\emph{An introduction to variational inequalities and their applications},
Academic Press (1980).

\bibitem{KuWe} P. C. Kunstmann, L. Weis. Maximal $L_p$-regularity for parabolic equations, Fourier 
multiplier theorems and $H^\infty$-calculus, in "Functional Analytic Methods for Evolution Equations"
 (eds. M. Iannelli, R. Nagel and S. Piazzera), Springer Lecture Notes 1855, 65--311 (2004).

\bibitem{Lamberton} D. Lamberton. \'{E}quations d'\'evolution lin\'eaires associ\'ees \`a des
semi-groupes de contractions dans les espaces {$L^p$}. \emph{J. Funct. Anal.} 72(2) (1987), 252--262.

\bibitem{Merdy} C. Le Merdy. $H^\infty$-functional calculus and applications to maximal regularity,
 \emph{Publ. Math. Besancon} 16 (1998), 41--77.

\bibitem{LX} C. Le Merdy, Q. Xu. Maximal theorems and square functions for analytic operators on 
$L^p$-spaces, \emph{J. London Math. Soc.} 86(2) (2012), 343--365.

\bibitem{luna}
A.~Lunardi: \emph{Analytic semigroups and optimal regularity in parabolic
problems}. Vol.~16 of Progress in Nonlinear Differential Equations and their
Applications, Birkh{\"a}user Verlag, Basel, 1995.

\bibitem{MMS} L. Maniar, M. Meyries, R. Schnaubelt. \emph{Null controllability for parabolic 
equations with dynamic boundary conditions of reactive-diffusive type}, Preprint, arXiv:1311.0761.

\bibitem{MM} M. Marcus, V. J. Mizel. 
Every superposition operator mapping one {S}obolev space to another is continuous,
\emph{J. Funct. Anal.} 33 (1979), 217--229.

\bibitem{MMT} D. Mitrea, M. Mitrea, and M. Taylor.
 Layer potentials, the {H}odge {L}aplacian, and global boundary
              problems in nonsmooth {R}iemannian manifolds,
\emph{Mem. Amer. Math. Soc.}, 150(713) (2001), x+120.

\bibitem{NS} J. Naumann, C. Simader. 
\emph{Measure and integration on Lipschitz manifolds}, Preprint, HU Berlin (2007).

\bibitem{Nittka} R. Nittka. Quasilinear elliptic and parabolic {R}obin problems on 
Lipschitz domains, \emph{NoDEA Nonlinear Differential Equations Appl.} 20(3) (2013), 1125--1155.

\bibitem{Nittka2} R. Nittka. Regularity of solutions of linear second order elliptic and 
parabolic boundary value problems on {L}ipschitz domains, 
\emph{J. Differential Equations} 251(4-5) (2011), 860--880.

\bibitem{Ou} E. M. Ouhabaz. \emph{Analysis of heat equations on domains}, 
Princeton University Press (2004).
               
\bibitem{Pruess} J. Pr\"uss. Maximal regularity for evolution equations in $L_p$-spaces. 
\emph{Conf. Semin. Mat. Univ. Bari} 285 (2003), 1--39.  
              


\bibitem{Schumacher} K. Schumacher. A chart preserving the normal vector and extensions of 
normal derivatives in weighted function spaces. \emph{Czechoslovak Math. J.} 59(134) (2009), 637--648.

\bibitem{Simon} L. Simon \emph{Lectures on geometric measure theory},
Proceedings of the Centre for Mathematical Analysis, Vol. 3, Australian National University (1983).
              
\bibitem{SW} J. Sprekels, H. Wu. A note on parabolic equation with nonlinear dynamical boundary condition, 
\emph{Nonlinear Analysis} 72 (2006), 3028--3048. 

\bibitem{tamm}
{Tamm, I.~E.}, \emph{Fundamentals of the theory of electricity}. Mir Publishers, Moscow, 1979.

           
\bibitem{Triebel} H. Triebel. \emph{Interpolation theory, function spaces, differential operators},
Johann Ambrosius Barth, 1995.
            
\bibitem{VaVi} J. L. V\'azquez, E. Vitillaro. Heat equation with dynamical boundary conditions of 
reactive-diffusive type, \emph{J. Differential Equations} 250 (2011), 2143--2161.

\bibitem{VV} H. Vogt, J. Voigt. Wentzell boundary conditions in the context of Dirichlet forms. 
\emph{Adv. Diff. Eq.} 8 (2003), 821--842.

\bibitem{Weis} L. Weis. A new approach to maximal $L^p$-regularity, pp. 195--214 in 
"Proc. of the 6th International Conference on Evolution Equations 1998", edited by 
G. Lumer and L. Weis, Marcel Dekker, 2000.

\bibitem{wloka}
J.~Wloka, \emph{Partial differential equations}, Cambridge University Press,
 Cambridge 1987

\end{thebibliography}
\end{document}